\documentclass{article}
\usepackage[utf8]{inputenc}
\usepackage{mathtools}
\usepackage{amsfonts}
\usepackage{amssymb}
\usepackage{dsfont}
\usepackage{mathrsfs}
\usepackage{yhmath}
\usepackage{hyperref}
\usepackage{stmaryrd}
\usepackage{amsthm}
\newtheorem{theorem}{Theorem}[section]
\newtheorem{lemma}[theorem]{Lemma}
\newtheorem{definition}[theorem]{Definition}
\newtheorem{corollary}[theorem]{Corollary}
\newtheorem{proposition}[theorem]{Proposition}

\title{Solving the Cahn-Hilliard equation with additive noise}
\author{Joe Ghafari}

\begin{document}

\maketitle

\begin{abstract}
We prove local well-posedness of the Cahn-Hilliard equation with additive noise.
Our method relies on paracontrolled calculus and the Da Prato-Debussche trick.
\end{abstract}



\section{Introduction}
Unless otherwise stated, in this paper we fix $d \in \mathbb{N}^*$ and we denote by $\mathbb{T}^d$ the $d$-dimensional torus.

The stochastic Cahn-Hilliard equation on $\mathbb{T}^d$ is given by: $$\partial_r f+\Delta^2 f=\Delta (f^3-f)+\xi,$$ where $\xi$ is a space-time white noise.

In recent years, stochastic partial differential equations have been under intensive study.

Da Prato and Debussche \cite{MR2016604} first solved the $2$-dimensional stochastic quantization equation. Hairer \cite{MR3071506} proved local well-posedness of the KPZ equation using rough paths theory. Hairer \cite{MR3274562} solved the $3$-dimensional stochastic quantization equation employing his new theory of regularity structures. Catellier and Chouk \cite{MR3846835} proved the existence and uniqueness of a local-in-time solution to the $3$-dimensional stochastic quantization using paracontrolled distributions developped recently by Gubinelli, Imkeller, and Perkowski \cite{MR3406823}. Gubinelli and Perkowski \cite{MR3592748} solved the KPZ equation using paracontrolled calculus.

The purpose of this article is to prove local well-posedness of the stochastic Cahn-Hilliard equation in space dimensions $1,2,3,$ and $4$ using the Da Prato-Debussche trick and paracontrolled calculus.

The reader is referred to \cite{MR1359472} for another proof of the existence and uniqueness of a solution to the stochastic Cahn-Hilliard equation in space dimensions $1,2,$ and $3.$

The article is organized as follows. In section \ref{section2} we introduce Besov spaces and their properties. In section \ref{section3} we provide some probabilistic preliminaries. Finally, in section \ref{section4} we solve the stochastic Cahn-Hilliard equation. 
\section{Besov spaces and paraproducts}
\label{section2}
In this section we recall the definitions and the basic results of paracontrolled calculus (see \cite{MR2768550,MR3406823,MR3445609} for the proofs).
\subsection{Tempered distributions}
\begin{definition}
Let $(q,m) \in [1,+\infty[\times \mathbb{N}^*.$
\begin{enumerate}
\item We denote by $L^2((\mathbb{R}_+\times \mathbb{T}^d)^m)$ the set of functions $g:(\mathbb{R}_+\times \mathbb{T}^d)^m\longrightarrow\mathbb{R}$ such that $g$ is a $(\mathcal{B}((\mathbb{R}_+\times \mathbb{T}^d)^m),\mathcal{B}(\mathbb{R}))$-measurable function and $$\int_{(\mathbb{R}_+\times \mathbb{T}^d)^m}\!|g(x)|^2\,\mathrm{d}x<+\infty.$$
\item We denote by $L^q(\mathbb{T}^d)$ the set of functions $g:\mathbb{T}^d\longrightarrow\mathbb{R}$ such that $g$ is a $(\mathcal{B}(\mathbb{T}^d),\mathcal{B}(\mathbb{R}))$-measurable function and $$\int_{\mathbb{T}^d}\!|g(x)|^q\,\mathrm{d}x<+\infty.$$
We also define
\begin{align*} 
\Vert\cdot\Vert_{L^q(\mathbb{T}^d)}: L^q(\mathbb{T}^d) &\longrightarrow \mathbb{R} \\ h & \longmapsto \Vert h\Vert_{L^q(\mathbb{T}^d)}=(\int_{\mathbb{T}^d}\!|h(x)|^q\,\mathrm{d}x)^{\frac{1}{q}}
\end{align*}
\item We denote by $L^\infty(\mathbb{T}^d)$ the set of functions $g:\mathbb{T}^d\longrightarrow\mathbb{R}$ such that $g$ is a $(\mathcal{B}(\mathbb{T}^d),\mathcal{B}(\mathbb{R}))$ and there exists $c \in \mathbb{R}_+,|g(x)|\leq c$ $\mathrm{d}x$-almost everywhere on $\mathbb{T}^d.$
\\
We also define
\begin{align*} 
\Vert\cdot\Vert_{L^\infty(\mathbb{T}^d)}: L^\infty(\mathbb{T}^d) &\longrightarrow \mathbb{R} \\ h & \longmapsto \Vert h\Vert_{L^\infty(\mathbb{T}^d)}=\mathrm{ess\,sup}_{x \in \mathbb{T}^d}|h(x)|
\end{align*}
\end{enumerate}
\end{definition}
\begin{definition}
We write $C^{\infty}(\mathbb{T}^d,\mathbb{R})$ for the set of functions $h:\mathbb{T}^d\longrightarrow\mathbb{R}$ such that $h$ is infinitely differentiable on $\mathbb{T}^d$ and we denote by $C_c^{\infty}(\mathbb{R}^d,\mathbb{R})$ the set of functions $g:\mathbb{R}^d\longrightarrow\mathbb{R}$ such that $g$ is infinitely differentiable on $\mathbb{R}^d$ and $\overline{\{x \in \mathbb{R}^d,g(x)\neq 0\}}$ is a compact set in $\mathbb{R}^d.$
\end{definition}
\begin{definition}
For all $j \in \mathbb{N},$ let
\begin{align*} 
\Vert\cdot\Vert_{j,\mathscr{S}}: C^{\infty}(\mathbb{T}^d,\mathbb{R})&\longrightarrow \mathbb{R} \\ \phi & \longmapsto \Vert \phi\Vert_{j,\mathscr{S}}=\max_{\substack{\mu \in \mathbb{N}^d\\ |\mu|\leq j}}\sup_{x \in \mathbb{T}^d}|\partial^\mu\phi(x)|
\end{align*}
\begin{enumerate}
\item The Schwartz space is the vector space $\mathscr{S}:=C^\infty(\mathbb{T}^d,\mathbb{R})$ equipped with the topology generated by the family of semi-norms $(\Vert\cdot\Vert_{j,\mathscr{S}})_{j \in \mathbb{N}}.$
\item The space of tempered distributions is the vector space $\mathscr{S}':=\{f \in \mathbb{R}^{\mathscr{S}},\exists (c,q) \in \mathbb{R}_+\times \mathbb{N},\forall \phi \in \mathscr{S},|f(\phi)|\leq c\Vert\phi\Vert_{q,\mathscr{S}}\}\cap\{f \in \mathbb{R}^{\mathscr{S}},\forall (\alpha,\phi_1\phi_2) \in \mathbb{R}\times\mathscr{S}^2,f(\alpha\phi_1+\phi_2)=\alpha f(\phi_1)+f(\phi_2)\}$ equipped with the weak$^*$ topology $\sigma(\mathscr{S}',\mathscr{S}).$
\item Let $(f_j)_{j \in \mathbb{N}}$ be a sequence in $\mathscr{S}'.$ We say that $(f_j)_{j \in \mathbb{N}}$ converges in $\mathscr{S}'$ if and only if there exists $f \in \mathscr{S}'$ such that for all $\phi \in \mathscr{S},\lim_{q\to+\infty}f_q(\phi)=f(\phi).$
\end{enumerate}
\end{definition}
\begin{definition}
Let $f \in \mathscr{S}'$ and $h:\mathbb{Z}^d\longrightarrow\mathbb{C}$ be a function such that $$\exists (c,q) \in \mathbb{R}_+\times \mathbb{N},\forall x \in \mathbb{Z}^d,|h(x)|\leq c(1+|x|)^q.$$
\begin{enumerate}
\item The Fourier transform of $f$ is given by
\begin{align*}
\mathscr{F}f: \mathbb{Z}^d &\longrightarrow \mathbb{C} \\ 
m & \longmapsto \mathscr{F}f=f(\cos(2\pi\langle \mathrm{Id}_{\mathbb{T}^d},m\rangle))-\mathrm{i}f(\sin(2\pi\langle\mathrm{Id}_{\mathbb{T}^d},m\rangle))
\end{align*}
\item The inverse Fourier transform of $h$ is given by
\begin{align*} 
\mathscr{F}^{-1}h: \mathscr{S} &\longrightarrow \mathbb{C} \\ \phi & \longmapsto \mathscr{F}^{-1}h(\phi)=\sum_{n \in \mathbb{Z}^d}h(n)\int_{\mathbb{T}^d}\!\phi(x)\mathrm{e}^{2\pi\mathrm{i}\langle n,x\rangle}\,\mathrm{d}x
\end{align*}
\end{enumerate}
\end{definition}
\subsection{Besov spaces}
\begin{definition}
A dyadic partition of unity consists of two non-negative radial functions $\rho_{-1}$ and $\rho_0$ such that $(\rho_{-1},\rho_{0}) \in (C_c^\infty(\mathbb{R}^d,\mathbb{R}))^2,\overline{\{y \in \mathbb{R}^d,\rho_{-1}(y)\neq 0\}}\subset\{y \in \mathbb{R}^d,|y|\leq \frac{4}{3}\},\overline{\{y \in \mathbb{R}^d,\rho_{0}(y)\neq 0\}}\subset\{y \in \mathbb{R}^d,\frac{3}{4}\leq |y|\leq \frac{8}{3}\},$ and $$\forall x \in \mathbb{R}^d,\rho_{-1}(x)+\sum_{q \in \mathbb{N}}\rho_{0}(2^{-q}x)=1.$$
\end{definition}
For the rest of this article we fix a dyadic partition of unity $(\rho_{-1},\rho_0)$ and let
\begin{align*} 
K_{-1}: \mathbb{T}^d &\longrightarrow \mathbb{R}
\\ x & \longmapsto K_{-1}(x)=\sum_{m \in \mathbb{Z}^d}\rho_{-1}(m)\cos(2\pi\langle x,m\rangle)
\end{align*}
For every $q \in \mathbb{N},$ we define
\begin{align*} 
\rho_{q}: \mathbb{R}^d &\longrightarrow \mathbb{R} \\ x & \longmapsto \rho_{q}(x)=\rho_0(2^{-q}x)
\end{align*}
and
\begin{align*} 
K_{q}: \mathbb{T}^d &\longrightarrow \mathbb{R} \\ 
x & \longmapsto K_{q}(x)=\sum_{m \in \mathbb{Z}^d}\rho_q(m)\cos(2\pi\langle x,m\rangle)
\end{align*}
\begin{definition}
For every $(f,q) \in \mathscr{S}'\times (\mathbb{N}\cup\{-1\}),$ we define the Littlewood-Paley blocks of $f$ by $\Delta_{q}f:=K_q*f.$
\end{definition}
\begin{definition}
Let $(\alpha,\gamma,r)\in \mathbb{R}\times [1,+\infty[^2,$
\begin{align*} 
\Vert\cdot\Vert_{\alpha}: \mathscr{S}' &\longrightarrow \overline{\mathbb{R}}_+ \\ f & \longmapsto \Vert f\Vert_{\alpha}=\sup_{q \in \mathbb{N}}2^{\alpha(q-1)}\Vert \Delta_{q-1}f\Vert_{L^{\infty}(\mathbb{T}^d)}
\end{align*}
and
\begin{align*} 
\Vert\cdot\Vert_{B^{\alpha}_{r,\gamma}}: \mathscr{S}' &\longrightarrow \overline{\mathbb{R}}_+ \\ f & \longmapsto \Vert f\Vert_{B^{\alpha}_{r,\gamma}}=(\sum_{q \in \mathbb{N}}2^{\alpha\gamma(q-1)}\Vert \Delta_{q-1} f\Vert_{L^r(\mathbb{T}^d)}^\gamma)^{\frac{1}{\gamma}}
\end{align*}
We define $$B_{r,\gamma}^\alpha:=\{f \in \mathscr{S}',\Vert f\Vert_{B_{r,\gamma}^\alpha}<+\infty\}$$ and $$\mathscr{C}^\alpha:=\{f \in \mathscr{S}',\lim_{q\to+\infty}2^{\alpha q}\Vert\Delta_q f\Vert_{L^{\infty}(\mathbb{T}^d)}=0\}.$$
The Banach spaces $(B^\alpha_{r,\gamma},\Vert\cdot\Vert_{B^\alpha_{r,\gamma}})$ and $(\mathscr{C}^\alpha,\Vert\cdot\Vert_{\alpha})$ are called Besov spaces.
\end{definition}
\begin{definition}
Let $(\alpha,U) \in \mathbb{R}\times \mathbb{R}_+.$ We denote by $C_U\mathscr{C}^\alpha$ the set of functions $h:[0,U]\longrightarrow\mathscr{C}^\alpha$ such that $h$ is continuous on $[0,U].$
We also define
\begin{align*} 
\Vert\cdot\Vert_{C_U\mathscr{C}^\alpha}: C_U\mathscr{C}^\alpha &\longrightarrow \mathbb{R} \\ f & \longmapsto \Vert f\Vert_{C_U\mathscr{C}^\alpha}=\sup_{r \in [0,U]}\Vert f(r)\Vert_{\mathscr{C}^\alpha}
\end{align*}
\end{definition}
We note that for every $(\alpha,U) \in \mathbb{R}\times \mathbb{R}_+,(C_U\mathscr{C}^\alpha,\Vert\cdot\Vert_{C_U\mathscr{C}^\alpha})$ is a Banach space.

The following Bernstein inequalities are very useful when dealing with functions with compactly supported Fourier transform.
\begin{theorem}[Bernstein inequalities]
\label{Bernstein}
Let $\mu\in \mathbb{N}^d$ and $(\gamma_1,\gamma_2,\gamma_3) \in (\mathbb{R}_+^*)^3$ such that $\gamma_2<\gamma_3.$
\begin{enumerate}
\item There exists $c \in \mathbb{R}_+^*$ such that for any $(q,\beta) \in [1,+\infty[\times\mathbb{R}_+^*$ and all $f \in L^q(\mathbb{T}^d),$ if $\{x \in \mathbb{Z}^d,\int_{\mathbb{T}^d}\!f(y)\mathrm{e}^{-2\pi\mathrm{i}\langle x,y\rangle}\,\mathrm{d}y\neq 0\}\subset \{x \in \mathbb{R}^d,|x|\leq \beta\gamma_1\},$ then $$\Vert \partial^\mu f\Vert_{L^\infty(\mathbb{T}^d)}\leq c\beta^{\frac{d}{q}+|\mu|}\Vert f\Vert_{L^q(\mathbb{T}^d)}.$$
\item There exists $(c_1,c_2) \in (\mathbb{R}_+^*)^2$ such that for any $\beta \in \mathbb{R}_+^*$ and all $f \in L^\infty(\mathbb{T}^d),$ if $\{x \in \mathbb{Z}^d,\int_{\mathbb{T}^d}\!f(y)\mathrm{e}^{-2\pi\mathrm{i}\langle x,y\rangle}\,\mathrm{d}y\neq 0\}\subset\{x \in \mathbb{R}^d,\beta\gamma_2\leq |x|\beta\gamma_3\},$ then $$c_1\beta^{|\mu|}\Vert f\Vert_{L^\infty(\mathbb{T}^d)}\leq \Vert \partial^{\mu}f\Vert_{L^\infty(\mathbb{T}^d)}\leq c_2\beta^{|\mu|}\Vert f\Vert_{L^{\infty}(\mathbb{T}^d)}.$$
\end{enumerate}
\end{theorem}
The following two theorems are simple applications of Theorem \ref{Bernstein}.
\begin{theorem}[Besov embedding]
\label{embedding}
There exists $c \in \mathbb{R}_+^*$ such that $$\forall(f,\alpha,q) \in \mathscr{S}'\times \mathbb{R}\times [1,+\infty[,\Vert f\Vert_{\alpha-\frac{d}{q}}\leq c\Vert f\Vert_{B^\alpha_{q,d}}.$$
\end{theorem}
\begin{theorem}
\label{derivative}
For every $\mu \in \mathbb{N}^d,$ there exists $c \in \mathbb{R}_+^*$ such that $$\forall (f,\alpha) \in \mathscr{S}'\times \mathbb{R},\Vert\partial^\mu f\Vert_{\alpha-|\mu|}\leq c\Vert f\Vert_{\alpha}$$
\end{theorem}
\subsection{Paraproducts and resonant products}
The main difficulty is solving singular stochastic partial differential equations is to multiply distributions. Paraproducts are bilinear operations useful to decompose the multiplication into simpler problems. 
\begin{definition}
Let $(f,g) \in (\mathscr{S}')^2.$
\begin{enumerate}
\item If $(\sum_{j=1}^q\sum_{n=0}^{j-1}\Delta_{n-1}f\Delta_{j}g)_{q \in \mathbb{N}^*}$ converges in $\mathscr{S}',$ then the limit is denoted by $f\varolessthan g=g\varogreaterthan f=\sum_{q \in \mathbb{N}^*}\sum_{n=0}^{q-1}\Delta_{n-1}f\Delta_{q}g$ and it is called paraproduct. 
\item If $(\sum_{j=0}^q\Delta_{j-1}f(\Delta_{\max(j-2,-1)}g+\Delta_{j-1}g+\Delta_{j}g))_{q \in \mathbb{N}}$ converges in $\mathscr{S}',$ then the limit is denoted by $f\varodot g=\sum_{q \in \mathbb{N}}\Delta_{q-1}f(\Delta_{\max(q-2,-1)}g+\Delta_{q-1}g+\Delta_qg)$ and it is called resonant product.
\end{enumerate}
\end{definition}
Bony \cite{MR0631751} noticed that paraproducts are always well-defined distributions. The only problem in constructing the product of distributions is the resonant term. The following estimates gives the basic result about these bilinear operations.
\begin{theorem}
\begin{enumerate}
\item For all $\alpha \in \mathbb{R}^*,$ there exists $c \in \mathbb{R}_+^*$ such that for all $\beta \in \mathbb{R},$ $$\forall (f,h) \in \mathscr{C}^\alpha\times \mathscr{C}^\beta,\Vert f\varolessthan g\Vert_{\min(\alpha,0)+\beta}\leq c\Vert f\Vert_{\alpha}\Vert h\Vert_{\beta}.$$
\item For every $(\alpha,\beta) \in \mathbb{R}^2,$ if $\alpha+\beta>0,$ then $$\exists c \in \mathbb{R}_+^*,\forall (f,h) \in \mathscr{C}^\alpha\times \mathscr{C}^\beta,\Vert f\varodot h\Vert\leq c\Vert f\Vert_{\alpha}\Vert h\Vert_{\beta}.$$
\end{enumerate}
\end{theorem}
\begin{corollary}
\label{product}
For all $(\alpha,\beta) \in \mathbb{R}^2,$ if $\alpha+\beta>0,$ then there exists $c \in \mathbb{R}_+^*$ such that for every $(f,h) \in \mathscr{C}^\alpha\times \mathscr{C}^\beta,fh \in \mathscr{C}^{\min(\alpha,\beta)}$ and $$\Vert fh\Vert_{\min(\alpha,\beta)}\leq c\Vert f\Vert_{\alpha}\Vert h\Vert_{\beta}.$$
\end{corollary}
\subsection{Schauder estimates}
For all $(r,f)\in\mathbb{R}_+^*\times \mathscr{S}',$ let $P_0f:=f,$
\begin{align*} 
\eta_{r}: \mathbb{T}^d &\longrightarrow \mathbb{R} \\ 
x & \longmapsto \eta_{r}(x)=\sum_{q \in \mathbb{Z}^d}\cos(2\pi\langle x,q\rangle)\mathrm{e}^{-r|2\pi q|^4}
\end{align*}
and $P_rf:=\eta_r*f.$

We study here the regularizing effect of the semi-group $(P_q)_{q \in \mathbb{R}_+}.$

The following lemma is crucial for the proof of Theorem \ref{Schauder} below.
\begin{lemma}
\label{lemma}
\begin{enumerate}
\item For every $\beta \in \mathbb{R}_+,$ there exists $c \in \mathbb{R}_+^*$ such that for all $\alpha \in \mathbb{R},$ $$\forall (r,f) \in \mathbb{R}_+^*\times \mathscr{C}^\alpha,\Vert P_rf\Vert_{\alpha+\beta}\leq c(r)\max(r^{-\frac{\beta}{4}},1)\Vert f\Vert_{\alpha}.$$
\item For all $r \in \mathbb{R}_+^*,$ let $h \in C_c^\infty(\mathbb{R}^d,\mathbb{R})$ and 
\begin{align*}
h_r:\mathbb{T}^d&\longrightarrow\mathbb{R}\\
x&\longmapsto h_r(x)=\sum_{q \in \mathbb{Z}^d}h(rq)\cos(2\pi\langle x,q\rangle)
\end{align*}
such that $h(0)=1.$ For every $\alpha \in \mathbb{R},$ if $g \in \mathscr{C}^\alpha,$ then $\lim_{\epsilon\to0}\Vert h_\epsilon*g-g\Vert_{\alpha}=0$
\end{enumerate}
\end{lemma}
\begin{proof}
\begin{enumerate}
\item Let $\beta \in \mathbb{R}_+,q:=\left\lfloor \beta\right\rfloor,$ and $\mu \in \mathbb{N}^d$ such that $\mu=(q,0,...,0).$

It follows from Theorem \ref{Bernstein} that there exists $(c_1,c_2) \in (\mathbb{R}_+^*)^2$ such that for all $(\alpha,r) \in \mathbb{R}\times \mathbb{R}_+^*,$
\begin{align*}
\forall f \in \mathscr{C}^\alpha,\Vert P_r f\Vert_{\alpha+q}&\leq\max(2^{-\alpha}\Vert\Delta_{-1}P_rf\Vert_{L^\infty(\mathbb{T}^d)},\sup_{j \in \mathbb{N}}2^{j(\alpha+q)}\Vert\Delta_j P_rf\Vert_{L^{\infty}(\mathbb{T}^d)})\\
&\leq\max(2^{-\alpha}\Vert\eta_r\Vert_{L^1(\mathbb{T}^d)}\Vert\Delta_{-1}f\Vert_{L^\infty(\mathbb{T}^d)},c_1\sup_{j \in \mathbb{N}}2^{\alpha j}\Vert \partial^\mu\Delta_j P_rf\Vert_{L^\infty(\mathbb{T}^d)})\\
&\leq\max(\Vert f\Vert_{\alpha},c_1\sup_{j \in \mathbb{N}}2^{\alpha j}\Vert \partial^\mu\eta_{r}\Vert_{L^1(\mathbb{T}^d)}\Vert\Delta_jf\Vert_{L^{\infty}(\mathbb{T}^d)})\\
&\leq c_2\Vert f\Vert_{\alpha}\max(r^{-\frac{q}{4}},1).
\end{align*}
and $$\forall f \in \mathscr{C}^\alpha,\Vert P_rf\Vert_{\alpha+q+1}\leq c_2\Vert f\Vert_{\alpha}\max(r^{-\frac{q+1}{4}},1).$$
Since for any $\alpha \in \mathbb{R}$ and all $(j,r,f) \in \mathbb{N}\times \mathbb{R}_+^*\times \mathscr{C}^\alpha,$
\begin{align*}
\Vert\Delta_{j-1}P_rf\Vert_{L^\infty(\mathbb{T}^d)}&\leq c_2\Vert f\Vert_{\alpha}2^{-(j-1)(\alpha+q)}\min(2^{1-j}\max(r^{-\frac{q+1}{4}},1),\max(r^{-\frac{q}{4}},1))\\
&\leq c_2\Vert f\Vert_{\alpha}2^{-(j-1)(\alpha+q)}\max(r^{-\frac{q}{4}},1)\min(2^{1-j}r^{-\frac{1}{4}},1),
\end{align*}
We conclude that for every $\alpha \in \mathbb{R},$
\begin{align*}
\forall(r,f) \in \mathbb{R}_+^*\times \mathscr{C}^\alpha,\Vert P_rf\Vert_{\alpha+\beta} &\leq c_2\max(r^{-\frac{q}{4}},1)r^{\frac{q-\beta}{4}}\Vert f\Vert_{\alpha}\\
&\leq c_2\max(r^{-\frac{\beta}{4}},r^{\frac{q-\beta}{4}})\Vert f\Vert_{\alpha}\\
&\leq c_2r^{-\frac{\beta}{4}}\max(r^{\frac{\beta}{4}},1)\Vert f\Vert_{\alpha}\\
&\leq c_2\max(r^{-\frac{\beta}{4}},1)\Vert f\Vert_{\alpha}.
\end{align*}
\item Let $\alpha \in \mathbb{R},g \in \mathscr{C}^\alpha,$ and $(\epsilon_q)_{q \in \mathbb{N}}$ be a sequence in $\mathbb{R}_+^*$ such that $\lim_{q\to+\infty}\epsilon_q=0.$

For all $(n,j) \in \mathbb{N}^2,$ we have
\begin{align*}
\Vert h_{\epsilon_n}*g-g\Vert_{\alpha}&\leq\max_{0\leq m\leq j+1}2^{\alpha(m-1)}\Vert h_{\epsilon_n}*\Delta_{m-1}f-\Delta_{m-1} f\Vert_{L^{\infty}(\mathbb{T}^d)}\\
& \phantom{\leq}\ +\sup_{m \in \mathbb{N}\cap[j+1,+\infty[}2^{m\alpha}\Vert h_{\epsilon_n}*\Delta_mf-\Delta_mf\Vert_{L^{\infty}(\mathbb{T}^d)}\\
&\leq\max_{0 \leq m\leq j+1}2^{\alpha(m-1)}\Vert h_{\epsilon_n}*\Delta_{m-1}f-\Delta_{m-1}f\Vert_{L^\infty(\mathbb{T}^d)}\\
& \phantom{\leq}\ +(1+\int_{\mathbb{R}^d}\!|\int_{\mathbb{R}^d}\!h(y)\mathrm{e}^{2\pi\mathrm{i}\langle x,y\rangle}\,\mathrm{d}y|\,\mathrm{d}x)\sup_{m \in \mathbb{N}\cap[j+1,+\infty[}2^{m\alpha}\Vert\Delta_mf\Vert_{L^{\infty}(\mathbb{T}^d)}.
\end{align*}
Since for every $q \in \mathbb{N},\Delta_{q-1}f$ is uniformly continuous on $\mathbb{T}^d,$ it follows that $$\forall j \in \mathbb{N},\limsup_{n\to+\infty}\Vert h_{\epsilon_n}*g-g\Vert_{\alpha}\leq (1+\int_{\mathbb{R}^d}\!|\int_{\mathbb{R}^d}\!h(y)\mathrm{e}^{2\pi\mathrm{i}\langle x,y\rangle}\,\mathrm{d}y|\,\mathrm{d}x)\sup_{m \in \mathbb{N}\cap [j+1,+\infty[}2^{m\alpha}\Vert \Delta_mf\Vert_{L^\infty(\mathbb{T}^d)}.$$
Taking $j\to+\infty$ we deduce that $\lim_{n\to+\infty}\Vert h_{\epsilon_n}*g-g\Vert_{\alpha}=0.$
\end{enumerate}
\end{proof}
We derive next Schauder estimates for the semi-group $(P_r)_{r \in \mathbb{R}_+}.$
\begin{theorem}[Schauder estimates]
\label{Schauder}
\begin{enumerate}
\item For every $\beta \in ]-\infty,4[,$ there exists $c \in \mathbb{R}_+^*$ such that for all $(\alpha,U) \in \mathbb{R}\times \mathbb{R}_+,$ $$\forall g \in C_{U}\mathscr{C}^\alpha,\sup_{r \in [0,U]}\Vert\int_{0}^r\!P_{r-q}(g(q))\,\mathrm{d}q\Vert_{\alpha+\beta}\leq c\max(U,U^{1-\frac{\beta}{4}})\Vert g\Vert_{\alpha}.$$
\item Let $\gamma \in \mathbb{R},f\in \mathscr{C}^\gamma,$ and 
\begin{align*}
h:\mathbb{R}_+&\longrightarrow\mathscr{S}'\\
r&\longmapsto h(r)=P_rf
\end{align*}
There exists $c \in \mathbb{R}_+^*$ such that for every $U \in \mathbb{R}_+,h\in C_U\mathscr{C}^\gamma$ and $\Vert h\Vert_{C_U\mathscr{C}^\gamma}\leq c\Vert f\Vert_{\alpha}.$
\end{enumerate}
\end{theorem}
\begin{proof}
\begin{enumerate}
\item Applying Lemma \ref{lemma} we deduce that for every $\beta \in ]-\infty,4[,$ there exists $c \in \mathbb{R}_+^*$ such that for all $(\alpha,U)\in\mathbb{R}\times \mathbb{R}_+,$
\begin{align*}
\forall g \in C_U\mathscr{C}^\alpha,\sup_{r \in [0,U]}\Vert \int_0^r\!P_{r-q}(g(q))\,\mathrm{d}q\Vert_{\alpha+\beta}&\leq\sup_{r \in [0,U]}\int_{0}^r\!\Vert P_{r-q}(g(q))\Vert_{\alpha+\beta}\,\mathrm{d}q\\
&\leq c \sup_{r \in [0,U]}\int_{0}^r\max(|r-q|^{\frac{\beta}{4}},1)\Vert g\Vert_{C_U\mathscr{C}^\alpha}\\
&\leq c\max(U,U^{1-\frac{\beta}{4}})\Vert g\Vert_{C_U\mathscr{C}^\alpha}.
\end{align*}
\item It follows from Lemma \ref{lemma} that there exists $c \in \mathbb{R}_+^*$ such that $$\forall (r,g) \in \mathbb{R}_+\times \mathscr{C}^\gamma,\Vert P_rg\Vert_{\gamma}\leq c\Vert g\Vert_{\gamma}.$$
Let $U \in \mathbb{R}_+,r \in [0,U],$ and $(q_n)_{n \in \mathbb{N}}$ be a sequence in $[0,U]$ such that $\lim_{j\to+\infty}q_j=r.$

We have $$\forall n \in \mathbb{N},\Vert P_{q_n}f-P_rf\Vert_{\gamma}\leq c\Vert P_{|q_n-r|}f-f\Vert_{\gamma}.$$
Taking $n\to+\infty$ and applying Lemma \ref{lemma}, we conclude that $h$ is continuous on $[0,U].$
\end{enumerate}
\end{proof}
\section{Stochastic convolution and its Wick powers}
\label{section3}
In this section we provide the necessary probabilistic tools to solve the stochastic Cahn-Hilliard equation.

For the rest of this article we fix a complete probability space $(\Omega,\mathcal{D},\mathds{P})$ and we denote by $L^2(\Omega)$ the set of random variables $g:\Omega\longrightarrow\mathbb{R}$ such that $\mathds{E}[|g|^2]<+\infty.$

We start by recalling the definition of a space-time white noise.
\begin{definition}
A space-time white noise is a family $(\xi_\phi)_{\phi \in L^2(\mathbb{R}_+\times \mathbb{T}^d)}$ of real random variables such that for every $h \in L^2(\mathbb{R}_+\times \mathbb{T}^d),\xi_h$ is normally distributed, $\mathds{E}[\xi_h]=0,$ and $\mathds{E}[\xi_h^2]=\int_{\mathbb{R}_+\times \mathbb{T}^d}\!|h(x)|^2\,\mathrm{d}x.$
\end{definition}
For all $r \in \mathbb{R}_+,$ let $$\mathcal{A}_r:=\{\phi \in L^2(\mathbb{R}_+\times \mathbb{T}^d),\forall (q,x) \in ]r,+\infty[\times \mathbb{T}^d,\phi(q,x)=0\}$$ and $$\mathcal{D}_r:=\bigcap_{q \in ]r,+\infty[}\sigma(\{M \in \mathcal{D},\mathds{P}(M)=0\}\cup \sigma(\bigcup_{\phi \in \mathcal{A}_q}\sigma(\xi_\phi))).$$
Unless otherwise stated, all stochastic processes throughout are defined on the filtered probability space $(\Omega,\mathcal{D},(\mathcal{D}_r)_{r \in \mathbb{R}_+},\mathds{P}).$

The following estimate is needed to control high moments of Wiener-It\^o integrals.
\begin{theorem}[Nelson estimate]
\label{nelson}
For all $q \in \mathbb{N}^*,$ let
\begin{align*}
H_q:L^2((\mathbb{R}_+\times\mathbb{T}^d)^q)&\longrightarrow L^2(\Omega)\\
\phi&\longmapsto H_q(\phi)=\int_{(\mathbb{R}_+\times \mathbb{T}^d)^q}\!\phi(x_1,...,x_q)\,\xi(\mathrm{d}x_1)\,...\,\xi(\mathrm{d}x_q)
\end{align*}
For every $(q,r) \in \mathbb{N}^*\times \mathbb{R}_+^*,$ there exists $c \in \mathbb{R}_+^*$ such that $$\forall \phi \in L^2((\mathbb{R}_+\times \mathbb{T}^d)^q),\mathds{E}[|H_q(\phi)|^r]^{\frac{2}{r}}\leq c\mathds{E}[|H_q(\phi)|^2].$$
\end{theorem}
We refer to \cite{MR2200233} for a proof of Theorem \ref{nelson}.

Let $\zeta \in C_c^\infty(\mathbb{R}^d,\mathbb{R})$ such that $\zeta(0)=1$ and $$\forall x \in \mathbb{R}^d,\zeta(-x)=\zeta(x).$$
For all $(r,\epsilon) \in \mathbb{R}_+\times \mathbb{R}_+^*,$ let
\begin{align*}
X_r:L^2(\mathbb{T}^d)&\longrightarrow L^2(\Omega)\\
\phi&\longmapsto X_r(\phi)=\int_{[0,r]\times \mathbb{T}^d}\!(\eta_{r-q}*\phi)(x)\,\xi(\mathrm{d}q,\mathrm{d}x)
\end{align*}
and 
\begin{align*}
\zeta_\epsilon:\mathbb{T}^d&\longrightarrow\mathbb{R}\\
x&\longmapsto\zeta_\epsilon(x)=\sum_{q \in \mathbb{Z}^d}\zeta(\epsilon q)\cos(2\pi\langle x,q\rangle)
\end{align*}
The following version of Kolmogorov continuity theorem is needed for the construction of the stochastic convolution $(X_r)_{r \in \mathbb{R}_+}$ and its Wick powers in Besov spaces.
\begin{theorem}
\label{kolmogorov}
Let $(f_r)_{r \in \mathbb{R}_+}$ be a stochastic process of state space $(\mathscr{S}',\mathcal{B}(\mathscr{S}'))$ and $(\alpha,\beta) \in \mathbb{R}^2$ such that $\beta<\alpha.$ If there exists $(c,\gamma,q) \in (\mathbb{R}_+^*)^2\times ]\frac{\alpha-\beta}{2d},+\infty[$ such that for every $(r_1,r_2,j,x) \in (\mathbb{R}_+)^2\times \mathbb{N}\times \mathbb{T}^d,$ $$\mathds{E}[|\Delta_{j-1}f_0(x)|^q]^{\frac{1}{q}}+\mathds{E}[|\Delta_{j-1}f_{r_2}(x)-\Delta_{j-1}f_{r_1}(x)|^q]^{\frac{1}{q}}\leq c2^{-\alpha(j-1)}|r_2-r_1|^{\gamma+\frac{1}{q}},$$ then there exists a stochastic process $(g_r)_{r \in \mathbb{R}_+}$ of state space $(\mathscr{C}^\beta,\mathcal{B}(\mathscr{C}^\beta))$ such that the sample paths of $(g_q)_{q \in \mathbb{R}_+}$ are continuous and for all $r \in \mathbb{R}_+,\mathds{P}(g_r=f_r)=1.$
\end{theorem}
\begin{proof}
Let $m:=\frac{\alpha+\beta}{2}.$

There exists $(c_1,c_2,c_3,c_4,\gamma,q)\in (\mathbb{R}_+^*)^5\times ]\frac{\alpha-\beta}{2d},+\infty[$ such that for every $(r_1,r_2) \in (\mathbb{R}_+)^2,$
\begin{align*}
\mathds{E}[\Vert f_{r_2}-f_{r_1}\Vert_{\beta}]&\leq c_1\mathds{E}[\Vert f_{r_2}-f_{r_1}\Vert_{m-\frac{d}{q}}]\\
&\leq c_2\mathds{E}[\Vert f_{r_2}-f_{r_1}\Vert_{B^{m}_{q,q}}]\\
&\leq c_2\sum_{j \in \mathbb{N}}2^{qm(j-1)}\int_{\mathbb{T}^d}\!\mathds{E}[|\Delta_{j-1}f_{r_2}(x)-\Delta_{j-1}f_{r_1}(x)|^q]\,\mathrm{d}x\\
&\leq c_3|r_2-r_1|^{\gamma q+1}\sum_{j \in \mathbb{N}}2^{q(j-1)(m-\alpha)}\\
&\leq c_4|r_2-r_1|^{\gamma q+1},
\end{align*}
completing the proof.
\end{proof}
For all $r \in \mathbb{R}_+,$ let
\begin{align*}
Y_r:L^2(\mathbb{T}^d)&\longrightarrow L^2(\Omega)\\
\phi&\longmapsto Y_r(\phi)=\int_{([0,r]\times \mathbb{T}^d)^2}\!(\int_{\mathbb{T}^d}\!\phi(x)\prod_{m=1}^2\eta_{r-q_m}(x-y_m)\,\mathrm{d}x)\,\xi(\mathrm{d}q_1,\mathrm{d}y_1)\,\xi(\mathrm{d}q_2,\mathrm{d}y_2)
\end{align*}
and 
\begin{align*}
G_r:L^2(\mathbb{T}^d)&\longrightarrow L^2(\Omega)\\
\phi&\longmapsto G_r(\phi)=\int_{([0,r]\times \mathbb{T}^d)^3}\!(\int_{\mathbb{T}^d}\!\phi(x)\prod_{m=1}^3\eta_{r-q_m}(x-y_m)\,\mathrm{d}x)\,\xi(\mathrm{d}q_1,\mathrm{d}y_1)\,\xi(\mathrm{d}q_2,\mathrm{d}y_2)\,\xi(\mathrm{d}q_3,\mathrm{d}y_3)
\end{align*}
We provide next the optimal regularity properties of $(X_r)_{r \in \mathbb{R}_+}$ and its Wick powers.
\begin{theorem}
\label{theorem1}
Let $\theta \in ]-\infty,2-\frac{d}{2}[.$
\begin{enumerate}
\item There exists a stochastic process $(\widetilde{X}_r)_{r \in \mathbb{R}_+}$ of state space $(\mathscr{C}^\theta,\mathcal{B}(\mathscr{C}^\theta))$ such that the sample paths of $(\widetilde{X}_q)_{q \in \mathbb{R}_+}$ are continuous and for all $(r,\phi) \in \mathbb{R}_+\times \mathscr{S},\mathds{P}(\widetilde{X}_r(\phi)=X_r(\phi))=1.$
\item For all $(r,q) \in \mathbb{R}_+\times \mathbb{R}_+^*,$ let $\widetilde{X}_{r,q}:=\zeta_q*\widetilde{X}_r.$ For every $(U,\sigma) \in \mathbb{R}_+\times \mathbb{R}_+^*,$ $$\lim_{\epsilon\to 0}\mathds{P}(\sup_{r \in [0,U]}\Vert \widetilde{X}_{r,\epsilon}-\widetilde{X}_r\Vert_{\theta}>\sigma)=0.$$
\end{enumerate}
\end{theorem}
\begin{proof}
\begin{enumerate}
\item We have $$\forall r \in \mathbb{R}_+,\sum_{m \in \mathbb{Z}^d}\frac{1}{(1+|m|)^{d+1}}\mathds{E}[|X_r(\cos(2\pi\langle\mathrm{Id}_{\mathbb{T}^d},m\rangle))|+|X_r(\sin(2\pi\langle\mathrm{Id}_{\mathbb{T}^d},m\rangle))|]<+\infty.$$

Therefore there exists a stochastic process $(\wideparen{X}_q)_{q \in \mathbb{R}_+}$ of state space $(\mathscr{S}',\mathcal{B}(\mathscr{S}'))$ such that for every $(r,\phi) \in \mathbb{R}_+\times \mathscr{S},\mathds{P}(\wideparen{X}_r(\phi)=X_r(\phi))=1.$

We fix $(\beta,\gamma) \in \mathbb{R}_+^*\times ]0,1[.$

It follows from theorem \ref{nelson} that there exists $(c_1,c_2,c_3) \in (\mathbb{R}_+^*)^3$ such that for all $(r_1,r_2,j,x) \in(\mathbb{R}_+)^2\times \mathbb{N}\times \mathbb{T}^d,$ if $r_1\leq r_2,$ then
\begin{align*}
\mathds{E}[|\Delta_{j-1}\wideparen{X}_{r_2}(x)-\Delta_{j-1}\wideparen{X}_{r_1}(x)|^\beta]^{\frac{2}{\beta}}&\leq c_1\mathds{E}[|\Delta_{j-1}\wideparen{X}_{r_2}(x)-\Delta_{j-1}\wideparen{X}_{r_1}(x)|^2]\\
&\leq c_1\int_{[0,r_1]\times\mathbb{T}^d}\!|\int_{\mathbb{T}^d}\!K_{j-1}(x-y)(\eta_{r_2-\alpha}(y-v)-\eta_{r_1-\alpha}(y-v))\,\mathrm{d}y|^2\\
&\phantom{\leq}\,\mathrm{d}\alpha\,\mathrm{d}v\\
& \phantom{\leq}\ +c_1\int_{[r_1,r_2]\times \mathbb{T}^d}\!|\int_{\mathbb{T}^d}\!K_{j-1}(x-y)\eta_{r_2-\alpha}(y-v)\,\mathrm{d}y|^2\,\mathrm{d}\alpha\,\mathrm{d}v\\
&\leq c_1\int_{[r_1,r_2]\times \mathbb{T}^d}\!|K_{j-1}*\eta_{r_2-\alpha}(y)-K_{j-1}*\eta_{r_1-\alpha}(y)|^2\,\mathrm{d}\alpha\,\mathrm{d}y\\
& \phantom{\leq}\ +c_1\int_{[r_1,r_2]\times \mathbb{T}^d}\!|K_{j-1}*\eta_{r_2-\alpha}(y)|^2\,\mathrm{d}\alpha\,\mathrm{d}y\\
&\leq c_1\sum_{m \in \mathbb{Z}^d}\int_{0}^{r_1}\!|\rho_{j-1}(m)|^2(\mathrm{e}^{-|2\pi m|^4(r_2-\alpha)}-\mathrm{e}^{-|2\pi m|^4(r_1-\alpha)})^2\,\mathrm{d}\alpha\\
& \phantom{\leq}\ +c_1\sum_{m \in \mathbb{Z}^d}\int_{0}^{r_2-r_1}\!|\rho_{j-1}(m)|^2\mathrm{e}^{-2\alpha|2\pi m|^4}\,\mathrm{d}\alpha\\
&\leq c_2|r_2-r_1|^{\gamma}\sum_{m \in \mathbb{Z}^d}|\rho_{j-1}(m)|^2(1+|m|^4)^{\gamma-1}\\
&\leq c_32^{(j-1)(d+4\gamma-4)}|r_2-r_1|^\gamma.
\end{align*}
Applying Theorem \ref{kolmogorov}, we conclude the proof.
\item Let $(\beta,\gamma) \in \mathbb{R}_+^*\times ]0,1[.$

It follows from theorem \ref{nelson} that there exists $(c_1,c_2,c_3) \in (\mathbb{R}_+^*)^3$ such that for all $(r_1,r_2,\epsilon,j,x) \in(\mathbb{R}_+)^2\times \mathbb{R}_+^*\times \mathbb{N}\times \mathbb{T}^d,$ if $r_1\leq r_2,$ then
\begin{align*}
\mathds{E}[|\Delta_{j-1}\widetilde{X}_{r_2,\epsilon}(x)-\Delta_{j-1}\widetilde{X}_{r_1}(x)\\
-\Delta_{j-1}\widetilde{X}_{r_1,\epsilon}(x)+\Delta_{j-1}\widetilde{X}_{r_1}(x)|^{\beta}]^{\frac{2}{\beta}}&\leq c_1\mathds{E}[|\Delta_{j-1}\widetilde{X}_{r_2,\epsilon}(x)-\Delta_{j-1}\widetilde{X}_{r_2}(x)\\
&\phantom{\leq}-\Delta_{j-1}\widetilde{X}_{r_1,\epsilon}(x)+\Delta_{j-1}\widetilde{X}_{r_1}(x)|^2]\\
&\leq c_1\sum_{m \in \mathbb{Z}^d}|\rho_{j-1}(m)|^2(\zeta(\epsilon m)-1)^2\\
&\phantom{\leq}\int_{0}^{r_1}\!(\mathrm{e}^{-|2\pi m|^4(r_2-\alpha)}-\mathrm{e}^{-|2\pi m|^4(r_1-\alpha)})^2\,\mathrm{d}\alpha\\
& \phantom{\leq}\ +c_1\sum_{m \in \mathbb{Z}^d}|\rho_{j-1}(m)|^2(\zeta(\epsilon m)-1)^2\int_{0}^{r_2-r_1}\!\mathrm{e}^{-2\alpha|2\pi m|^4}\,\mathrm{d}\alpha\\
&\leq c_2\epsilon^\gamma|r_2-r_1|^{\frac{\gamma}{2}}\sum_{m \in \mathbb{Z}^d}|\rho_{j-1}(m)|^2(1+|m|^{4})^{\gamma-1}\\
&\leq c_3\epsilon^\gamma2^{(j-1)(d+4\gamma-4)}|r_2-r_1|^{\frac{\gamma}{2}}.
\end{align*}
We deduce from Theorem \ref{kolmogorov} that $\lim_{\epsilon\to0}\mathds{E}[\sup_{r\in[0,U]}\Vert\widetilde{X}_{r,\epsilon}-\widetilde{X}_r\Vert_{\theta}]=0.$
\end{enumerate}
\end{proof}
The following lemma gives a bound on discrete convolutions.
\begin{lemma}
\label{lemma2}
Let $(\alpha,\beta) \in \mathbb{R}^2.$ If $\max(\alpha,\beta)<d$ and $\alpha+\beta>d,$ then $$\exists c \in \mathbb{R}_+^*,\forall q\in \mathbb{Z}^d,\sum_{m \in \mathbb{Z}^d}(1+|q-m|^4)^{-\frac{\alpha}{4}}(1+|m|^{4})^{-\frac{\beta}{4}}\leq c(1+|q|^4)^{\frac{d-\alpha-\beta}{4}}.$$
\end{lemma}
\begin{proof}
Let $$J_1:=\{(m_1,m_2) \in (\mathbb{Z}^d)^2,|m_1|\geq 2|m_1+m_2|\},$$ $$J_2:=\{(m_1,m_2) \in (\mathbb{Z}^d)^2,|m_1|\leq \frac{1}{2}|m_1+m_2|\},$$ $$J_3:=\{(m_1,m_2) \in(\mathbb{Z}^d)^2,|m_2|\leq\frac{1}{2}|m_1+m_2|\},$$ and $$J_4:=(\mathbb{Z}^d)^2-(J_1\cup J_2 \cup J_3).$$

There exists $(c_1,c_2) \in (\mathbb{R}_+^*)^2$ such that for all $q \in \mathbb{Z}^d,$
\begin{align*}
\sum_{m\in\mathbb{Z}^d}(1+|q-m|^4)^{-\frac{\alpha}{4}}(1+|m|^4)^{-\frac{\beta}{4}}&\leq\sum_{m\in \mathbb{Z}^d}(1+|q-m|^4)^{-\frac{\alpha}{4}}(1+|m|^4)^{-\frac{\beta}{4}}\mathds{1}_{J_1}(m,q-m)\\
& \phantom{\leq}\ +\sum_{m\in \mathbb{Z}^d}(1+|q-m|^4)^{-\frac{\alpha}{4}}(1+|m|^4)^{-\frac{\beta}{4}}\mathds{1}_{J_2}(m,q-m)\\
& \phantom{\leq}\ +\sum_{m\in \mathbb{Z}^d}(1+|q-m|^4)^{-\frac{\alpha}{4}}(1+|m|^4)^{-\frac{\beta}{4}}\mathds{1}_{J_3}(m,q-m)\\
& \phantom{\leq}\ +\sum_{m\in \mathbb{Z}^d}(1+|q-m|^4)^{-\frac{\alpha}{4}}(1+|m|^4)^{-\frac{\beta}{4}}\mathds{1}_{J_4}(m,q-m)\\
&\leq c_1\sum_{m\in \mathbb{Z}^d}(1+|m|^4)^{-\frac{\alpha+\beta}{4}}\mathds{1}_{J_1}(m,q-m)\\
& \phantom{\leq}\ +c_1(1+|q|^4)^{-\frac{\alpha}{4}}\sum_{m \in \mathbb{Z}^d}(1+|m|^4)^{-\frac{\beta}{4}}\mathds{1}_{J_2}(m,q-m)\\
& \phantom{\leq}\ +c_1(1+|q|^4)^{-\frac{\beta}{4}}\sum_{m \in \mathbb{Z}^d}(1+|q-m|^4)^{-\frac{\alpha}{4}}\mathds{1}_{J_3}(m,q-m)\\
& \phantom{\leq}\ +c_1(1+|q|^4)^{-\frac{\alpha}{4}}\sum_{m \in \mathbb{Z}^d}(1+|m|^4)^{-\frac{\beta}{4}}\mathds{1}_{J_4}(m,q-m)\\
&\leq c_2(1+|q|^4)^{\frac{d-\alpha-\beta}{4}}.
\end{align*}
\end{proof}
For all $(r,\epsilon) \in \mathbb{R}_+\times \mathbb{R}_+^*,$ let
\begin{align*}
Y_{r,\epsilon}:L^2(\mathbb{T}^d)&\longrightarrow L^2(\Omega)\\
\phi&\longmapsto Y_{r,\epsilon}(\phi)=\int_{([0,r]\times \mathbb{T}^d)^2}\!(\int_{\mathbb{T}^d}\!\phi(x)\prod_{m=1}^2(\zeta_\epsilon*\eta_{r-q_m})(x-y_m)\,\mathrm{d}x)\,\xi(\mathrm{dq_1,\mathrm{d}y_1})\,\xi(\mathrm{d}q_2,\mathrm{d}y_2)
\end{align*}
We note that for every $(r,\epsilon,\phi) \in \mathbb{R}_+\times \mathbb{R}_+^*\times \mathscr{S}$ and for $\mathds{P}$-almost every $\omega \in \Omega,$ $$(Y_{r,\epsilon}(\phi))(\omega)=\int_{\mathbb{T}^d}\!\phi(x)(((\widetilde{X}_{r,\epsilon}(\omega))(x))^2-\mathds{E}[|\widetilde{X}_{r,\epsilon}(x)|^2])\,\mathrm{d}x.$$
\begin{theorem}
\label{theorem2}
If $d \in \{4,5,6,7\},\theta \in ]-\infty,4-d[,$ then 
\begin{enumerate}
\item there exist a stochastic process $(\widetilde{Y}_r)_{r \in \mathbb{R}_+}$ of state space $(\mathscr{C}^\theta,\mathcal{B}(\mathscr{C}^\theta))$ such that the sample paths of $(\widetilde{Y}_q)_{q \in \mathbb{R}_+}$ are continuous and for all $(r,\phi) \in \mathbb{R}_+\times \mathscr{S},\mathds{P}(\widetilde{Y}_r(\phi)=Y_r(\phi))=1.$
\item for every $\epsilon \in \mathbb{R}_+^*,$ there exist a stochastic process $(\widetilde{Y}_{r,\epsilon})_{r \in \mathbb{R}_+}$ of state space $(\mathscr{C}^\theta,\mathcal{B}(\mathscr{C}^\theta))$ such that the sample paths of $(\widetilde{Y}_{q,\epsilon})_{q \in \mathbb{R}_+}$ are continuous and for all $(r,\phi) \in \mathbb{R}_+\times \mathscr{S},\mathds{P}(\widetilde{Y}_{r,\epsilon}(\phi)=Y_{r,\epsilon}(\phi))=1.$
\item for every $(U,\sigma) \in \mathbb{R}_+\times \mathbb{R}_+^*,$ $$\lim_{\epsilon\to0}\mathds{P}(\sup_{r \in [0,U]}\Vert\widetilde{Y}_{r,\epsilon}-\widetilde{Y}_r\Vert_{\theta}>\sigma)=0.$$
\end{enumerate}
\end{theorem}
\begin{proof}
\begin{enumerate}
\item We have $$\forall r \in \mathbb{R}_+,\sum_{m \in \mathbb{Z}^d}\frac{1}{(1+|m|)^{d+1}}\mathds{E}[|Y_r(\cos(2\pi\langle\mathrm{Id}_{\mathbb{T}^d},m\rangle))|+|Y_r(\sin(2\pi\langle\mathrm{Id}_{\mathbb{T}^d},m\rangle))|]<+\infty.$$
Therefore there exist stochastic processes $(\wideparen{Y}_q)_{q \in \mathbb{R}_+}$ of state space $(\mathscr{S}',\mathcal{B}(\mathscr{S}'))$ such that for every $(r,\phi) \in \mathbb{R}_+\times \mathscr{S},\mathds{P}(\wideparen{Y}_r(\phi)=Y_r(\phi))=1.$

Let $(\beta,\gamma) \in \mathbb{R}_+^*\times ]0,\frac{d}{8}[.$

It follows from Theorem \ref{nelson} and Lemma \ref{lemma2} that there exists $(c_1,c_2,c_3,c_4) \in (\mathbb{R}_+^*)^{4}$ such that for all $(r_1,r_2,j,x) \in (\mathbb{R}_+)^2\times\mathbb{N}\times \mathbb{T}^d,$ if $r_1\leq r_2,$ then  
\begin{align*}
\mathds{E}[|\Delta_{j-1}\wideparen{Y}_{r_2}(x)-\Delta_{j-1}\wideparen{Y}_{r_1}(x)|^\beta]^{\frac{2}{\beta}}&\leq c_1\mathds{E}[|\Delta_{j-1}\wideparen{Y}_{r_2}(x)-\Delta_{j-1}\wideparen{Y}_{r_1}(x)|^2]\\
&\leq c_1\int_{[0,r_1]^2}\!(\int_{(\mathbb{T}^d)^2}\!|\int_{\mathbb{T}^d}\!K_{j-1}(x-y)(\prod_{q=1}^2\eta_{r_2-\alpha_q}(y-v_q)\\
&\phantom{\leq}-\prod_{m=1}^2\eta_{r_1-\alpha_m}(y-v_m))\,\mathrm{d}y|^2\mathrm{d}v_1\,\mathrm{d}v_2)\,\mathrm{d}\alpha_1\,\mathrm{d}\alpha_2\\
& \phantom{\leq}\ +c_1\int_{[r_1,r_2]^2}\!(\int_{(\mathbb{T}^d)^2}\!|\int_{\mathbb{T}^d}\!K_{j-1}(x-y)\prod_{q=1}^2\eta_{r_2-\alpha_q}(y-v_q)\,\mathrm{d}y|^2\\
&\phantom{\leq}\mathrm{d}v_1\,\mathrm{d}v_2)\,\mathrm{d}\alpha_1\,\mathrm{d}\alpha_2\\
&\leq c_1\int_{[0,r_1]^2}\!(\int_{(\mathbb{T}^d)^4}\!\prod_{m=1}^2K_{j-1}(y_m)(\prod_{q=1}^2\eta_{r_2-\alpha_q}(y_m-v_q)\\
&\phantom{\leq}-\prod_{n=1}^2\eta_{r_1-\alpha_n}(y_m-v_n))\,\mathrm{d}v_1\,\mathrm{d}v_2\,\mathrm{d}y_1\,\mathrm{d}y_2)\,\mathrm{d}\alpha_1\,\mathrm{d}\alpha_2\\
& \phantom{\leq}\ +c_1\int_{[r_1,r_2]^2}\!(\int_{(\mathbb{T}^d)^4}\!\prod_{m=1}^2\prod_{q=1}^2K_{j-1}(y_m)\eta_{r_2-\alpha_q}(y_m-v_q)\\
&\phantom{\leq}\,\mathrm{d}v_1\,\mathrm{d}v_2\,\mathrm{d}y_1\,\mathrm{d}y_2)\,\mathrm{d}\alpha_1\,\mathrm{d}\alpha_2\\
&\leq c_1\int_{[0,r_1]^2}\!(\int_{(\mathbb{T}^d)^2}K_{j-1}(y_1)K_{j-1}(y_2)(\prod_{q=1}^2\eta_{2(r_2-\alpha_q)}(y_1-y_2)\\
&\phantom{\leq}-\prod_{m=1}^2\eta_{r_2+r_1-2\alpha_m}(y_1-y_2))\,\mathrm{d}y_1\,\mathrm{d}y_2)\,\mathrm{d}\alpha_1\,\mathrm{d}\alpha_2\\
& \phantom{\leq}\ +c_1\int_{[0,r_1]^2}\!(\int_{(\mathbb{T}^d)^2}K_{j-1}(y_1)K_{j-1}(y_2)(\prod_{q=1}^2\eta_{2(r_1-\alpha_q)}(y_1-y_2)\\
&\phantom{\leq}-\prod_{m=1}^2\eta_{r_2+r_1-2\alpha_m}(y_1-y_2))\,\mathrm{d}y_1\,\mathrm{d}y_2)\,\mathrm{d}\alpha_1\,\mathrm{d}\alpha_2\\
& \phantom{\leq}\ +c_1\int_{[r_1,r_2]^2}\!(\int_{(\mathbb{T}^d)^2}\!K_{j-1}(y_1)K_{j-1}(y_2)\prod_{q=1}^2\eta_{2(r_2-\alpha_q)}(y_1-y_2)\\
&\phantom{\leq}\,\mathrm{d}y_1\,\mathrm{d}y_2)\,\mathrm{d}\alpha_1\mathrm{d}\alpha_2\\
&\leq c_1\sum_{(m_1,m_2) \in (\mathbb{Z}^d)^2}|\rho_{j-1}(m_1+m_2)|^2\int_{[0,r_1]^2}\!(\prod_{q=1}^2\mathrm{e}^{-|2\pi m_q|^4(r_2-\alpha_q)}\\
&\phantom{\leq}-\prod_{n=1}^2\mathrm{e}^{-|2\pi m_n|^4(r_1-\alpha_n)})^2\,\mathrm{d}\alpha_1\,\mathrm{d}\alpha_2
\end{align*}
\begin{align*}
& +c_1\sum_{(m_1,m_2) \in (\mathbb{Z}^d)^2}|\rho_{j-1}(m_1+m_2)|^2\int_{[0,r_2-r_1]^2}\!\prod_{q=1}^2\mathrm{e}^{-2\alpha_q|2\pi m_q|^4}\,\mathrm{d}\alpha_1\,\mathrm{d}\alpha_2\\
&\leq c_2|r_2-r_1|^{2\gamma}\sum_{(m_1,m_2) \in (\mathbb{Z}^d)^2}|\rho_{j-1}(m_1+m_2)|^2\prod_{q=1}^2\frac{1}{(1+|m_q|^4)^{1-\gamma}}\\
&\leq c_2|r_2-r_1|^{2\gamma}\sum_{m_1 \in \mathbb{Z}^d}|\rho_{j-1}(m_1)|^2\sum_{m_2 \in \mathbb{Z}^d}(1+|m_1-m_2|^4)^{\gamma-1}(1+|m_2|^4)^{\gamma-1}\\
&\leq c_3|r_2-r_1|^{2\gamma}\sum_{m \in \mathbb{Z}^d}|\rho_{j-1}(m)|^2(1+|m|^4)^{\frac{d}{4}+2\gamma-2}\\
&\leq c_42^{(j-1)(2d+8\gamma-8)}|r_2-r_1|^{2\gamma}.
\end{align*}
Applying Theorem \ref{kolmogorov}, we complete the proof.
\item For all $\epsilon \in \mathbb{R}_+^*,$ there exist a stochastic process $(\wideparen{Y}_{q,\epsilon})_{q \in \mathbb{R}_+}$ of state space $(\mathscr{S}',\mathcal{B}(\mathscr{S}'))$ such that for every $(r,\phi) \in \mathbb{R}_+\times \mathscr{S},\mathds{P}(\wideparen{Y}_{r,\epsilon}(\phi)=Y_{r,\epsilon}(\phi))=1.$

We fix $(\beta,\gamma) \in \mathbb{R}_+^*\times ]0,\frac{d}{8}[.$

It follows from Theorem \ref{nelson} and Lemma \ref{lemma2} that there exists $(c_1,c_2,c_3) \in (\mathbb{R}_+^*)^3$ such that for all $(r_1,r_2,\epsilon,j,x) \in (\mathbb{R}_+)^2\times \mathbb{R}_+^*\times \mathbb{N}\times \mathbb{T}^d,$ if $r_1\leq r_2,$ then 
\begin{align*}
\mathds{E}[|\Delta_{j-1}\wideparen{Y}_{r_2,\epsilon}(x)-\Delta_{j-1}\widetilde{Y}_{r_2}(x)\\-\Delta_{j-1}\wideparen{Y}_{r_1,\epsilon}(x)+\Delta_{j-1}\widetilde{Y}_{r_1}(x)|^\beta]^{\frac{2}{\beta}}&\leq c_1\mathds{E}[|\Delta_{j-1}\wideparen{Y}_{r_2,\epsilon}(x)-\Delta_{j-1}\widetilde{Y}_{r_2}(x)-\Delta_{j-1}\wideparen{Y}_{r_1,\epsilon}(x)+\Delta_{j-1}\widetilde{Y}_{r_1}(x)|^2]\\
&\leq c_1\sum_{(m_1,m_2) \in (\mathbb{Z}^d)^2}|\rho_{j-1}(m_1+m_2)|^2\\
&\phantom{\leq}\int_{[0,r_1]^2}\!(\prod_{q=1}^2(\zeta(\epsilon m_q)-1)^2\mathrm{e}^{-|2\pi m_q|^4(r_2-\alpha_q)}\\
&\phantom{\leq}-\prod_{n=1}^2(\zeta(\epsilon m_n)-1)^2\mathrm{e}^{-|2\pi m_n|^4(r_1-\alpha_n)})^2\,\mathrm{d}\alpha_1\,\mathrm{d}\alpha_2\\
& \phantom{\leq}\ +c_1\sum_{(m_1,m_2) \in (\mathbb{Z}^d)^2}|\rho_{j-1}(m_1+m_2)|^2\\
&\phantom{\leq}\int_{[0,r_2-r_1]^2}\!\prod_{q=1}^2(\zeta(\epsilon m_q)-1)^2\mathrm{e}^{-2\alpha_q|2\pi m_q|^4}\,\mathrm{d}\alpha_1\,\mathrm{d}\alpha_2\\
&\leq c_2\epsilon^{2\gamma}|r_2-r_1|^{\gamma}\sum_{(m_1,m_2) \in (\mathbb{Z}^d)^2}|\rho_{j-1}(m_1+m_2)|^2\prod_{q=1}^2\frac{1}{(1+|m_q|^4)^{1-\gamma}}\\
&\leq c_32^{(j-1)(2d+8\gamma-8)}\epsilon^{2\gamma}|r_2-r_1|^{\gamma}
\end{align*}
Applying Theorem \ref{kolmogorov}, we conclude the proof.
\item We deduce from Theorem \ref{kolmogorov} that for every $U \in \mathbb{R}_+,\lim_{\epsilon\to0}\mathds{E}[\sup_{r \in [0,U]}\Vert\widetilde{Y}_{r,\epsilon}-\widetilde{Y}_r\Vert_{\theta}]=0.$
\end{enumerate}
\end{proof}
For all $(r,\epsilon) \in \mathbb{R}_+\times \mathbb{R}_+^*,$ let
\begin{align*}
G_{r,\epsilon}:L^2(\mathbb{T}^d)&\longrightarrow L^2(\Omega)\\
\phi&\longmapsto G_{r,\epsilon}(\phi)=\int_{([0,r]\times \mathbb{T}^d)^3}\!(\int_{\mathbb{T}^d}\!\phi(x)\prod_{m=1}^3(\zeta_\epsilon*\eta_{r-q_m})(x-y_m)\,\mathrm{d}x)\,\xi(\mathrm{d}q_1,\mathrm{d}y_1)\,\xi(\mathrm{d}q_2,\mathrm{d}y_2)\,\xi(\mathrm{d}q_3,\mathrm{d}y_3)
\end{align*}
We note that for every $(r,\epsilon,\phi) \in \mathbb{R}_+\times \mathbb{R}_+^*\times \mathscr{S}$ and for $\mathds{P}$-almost every $\omega \in \Omega,$ $$(G_{r,\epsilon}(\phi))(\omega)=\int_{\mathbb{T}^d}\!\phi(x)(((\widetilde{X}_{r,\epsilon}(\omega))(x))^3-3(\widetilde{X}_{r,\epsilon}(\omega))(x)\mathds{E}[|\widetilde{X}_{r,\epsilon}(x)|^2])\,\mathrm{d}x.$$
\begin{theorem}
\label{theorem3}
If $(d\in \{4,5\},\theta \in ]-\infty,6-\frac{3d}{2}[,$ then 
\begin{enumerate}
\item there exist a stochastic process $(\widetilde{G}_r)_{r \in \mathbb{R}_+}$ of state space $(\mathscr{C}^\theta,\mathcal{B}(\mathscr{C}^\theta))$ such that the sample paths of $(\widetilde{G}_q)_{q \in \mathbb{R}_+}$ are continuous and for all $(r,\phi) \in \mathbb{R}_+\times \mathscr{S},\mathds{P}(\widetilde{G}_r(\phi)=G_r(\phi))=1.$
\item for every $\epsilon \in \mathbb{R}_+^*,$ there exist a stochastic process $(\widetilde{G}_{r,\epsilon})_{r \in \mathbb{R}_+}$ of state space $(\mathscr{C}^\theta,\mathcal{B}(\mathscr{C}^\theta))$ such that the sample paths of $(\widetilde{G}_{q,\epsilon})_{q \in \mathbb{R}_+}$ are continuous and for all $(r,\phi) \in \mathbb{R}_+\times \mathscr{S},\mathds{P}(\widetilde{G}_{r,\epsilon}(\phi)=G_{r,\epsilon}(\phi))=1.$
\item for every $(U,\sigma) \in \mathbb{R}_+\times \mathbb{R}_+^*,$ $$\lim_{\epsilon\to0}\mathds{P}(\sup_{r \in [0,U]}\Vert\widetilde{G}_{r,\epsilon}-\widetilde{G}_r\Vert_{\theta}>\sigma)=0.$$
\end{enumerate}
\end{theorem}
\begin{proof}
\begin{enumerate}
\item We have $$\forall r \in \mathbb{R}_+,\sum_{m \in \mathbb{Z}^d}\frac{1}{(1+|m|)^{d+1}}\mathds{E}[|G_r(\cos(2\pi\langle\mathrm{Id}_{\mathbb{T}^d},m\rangle))|+|G_r(\sin(2\pi\langle\mathrm{Id}_{\mathbb{T}^d},m\rangle))|]<+\infty.$$
Therefore there exist a stochastic process $(\wideparen{G}_q)_{q \in \mathbb{R}_+}$ of state space $(\mathscr{S}',\mathcal{B}(\mathscr{S}'))$ such that for every $(r,\phi) \in \mathbb{R}_+\times \mathscr{S},\mathds{P}(\wideparen{G}_r(\phi)=G_r(\phi))=1.$

Let $(\beta,\gamma) \in \mathbb{R}_+^*\times ]0,1-\frac{d}{6}[.$

It follows from Theorem \ref{nelson} and Lemma \ref{lemma2} that there exists $(c_1,c_2,c_3,c_4,c_5) \in (\mathbb{R}_+^*)^{5}$ such that for all $(r_1,r_2,j,x) \in (\mathbb{R}_+)^2\times\mathbb{N}\times \mathbb{T}^d,$ if $r_1\leq r_2,$ then  
\begin{align*}
\mathds{E}[|\Delta_{j-1}\wideparen{G}_{r_2}(x)-\Delta_{j-1}\wideparen{G}_{r_1}(x)|^\beta]^{\frac{2}{\beta}}&\leq c_{1}\mathds{E}[|\Delta_{j-1}\wideparen{G}_{r_2}(x)-\Delta_{j-1}\wideparen{G}_{r_1}(x)|^2]\\
&\leq c_{1}\int_{[0,r_1]^3}\!(\int_{(\mathbb{T}^d)^3}\!|\int_{\mathbb{T}^d}\!K_{j-1}(x-y)(\prod_{q=1}^3\eta_{r_2-\alpha_q}(y-v_q)\\
&\phantom{\leq}-\prod_{m=1}^3\eta_{r_1-\alpha_m}(y-v_m))\,\mathrm{d}y|^2\mathrm{d}v_1\,\mathrm{d}v_2\,\mathrm{d}v_3)\,\mathrm{d}\alpha_1\,\mathrm{d}\alpha_2\,\mathrm{d}\alpha_3\\
& \phantom{\leq}\ +c_{1}\int_{[r_1,r_2]^3}\!(\int_{(\mathbb{T}^d)^3}\!|\int_{\mathbb{T}^d}\!K_{j-1}(x-y)\prod_{q=1}^3\eta_{r_2-\alpha_q}(y-v_q)\,\mathrm{d}y|^2\\
&\phantom{\leq}\mathrm{d}v_1\,\mathrm{d}v_2\,\mathrm{d}v_3)\,\mathrm{d}\alpha_1\mathrm{d}\alpha_2\,\mathrm{d}\alpha_3\\
&\leq c_{1}\int_{[0,r_1]^3}\!(\int_{(\mathbb{T}^d)^5}\!\prod_{m=1}^2K_{j-1}(y_m)(\prod_{q=1}^3\eta_{r_2-\alpha_q}(y_m-v_q)\\
&\phantom{\leq}-\prod_{n=1}^3\eta_{r_1-\alpha_n}(y_m-v_n))\,\mathrm{d}v_1\,\mathrm{d}v_2\,\,\mathrm{d}v_3\,\mathrm{d}y_1\,\mathrm{d}y_2)\,\mathrm{d}\alpha_1\,\mathrm{d}\alpha_2\,\mathrm{d}\alpha_3\\
& \phantom{\leq}\ +c_{1}\int_{[r_1,r_2]^3}\!(\int_{(\mathbb{T}^d)^5}\!\prod_{m=1}^2\prod_{q=1}^3K_{j-1}(y_m)\eta_{r_2-\alpha_q}(y_m-v_q)\\
&\phantom{\leq}\,\mathrm{d}v_1\,\mathrm{d}v_2\,\mathrm{d}v_3\,\mathrm{d}y_1\,\mathrm{d}y_2)\,\mathrm{d}\alpha_1\,\mathrm{d}\alpha_2\,\mathrm{d}\alpha_3\\
&\leq c_{1}\int_{[0,r_1]^3}\!(\int_{(\mathbb{T}^d)^2}K_{j-1}(y_1)K_{j-1}(y_2)(\prod_{q=1}^3\eta_{2(r_2-\alpha_q)}(y_1-y_2)\\
&\phantom{\leq}-\prod_{m=1}^3\eta_{r_2+r_1-2\alpha_m}(y_1-y_2))\,\mathrm{d}y_1\,\mathrm{d}y_2)\,\mathrm{d}\alpha_1\,\mathrm{d}\alpha_2\,\mathrm{d}\alpha_3\\
& \phantom{\leq}\ +c_{1}\int_{[0,r_1]^3}\!(\int_{(\mathbb{T}^d)^2}K_{j-1}(y_1)K_{j-1}(y_2)(\prod_{q=1}^3\eta_{2(r_1-\alpha_q)}(y_1-y_2)\\
&\phantom{\leq}-\prod_{m=1}^3\eta_{r_2+r_1-2\alpha_m}(y_1-y_2))\,\mathrm{d}y_1\,\mathrm{d}y_2)\,\mathrm{d}\alpha_1\,\mathrm{d}\alpha_2\,\mathrm{d}\alpha_3\\
& \phantom{\leq}\ +c_{1}\int_{[r_1,r_2]^3}\!(\int_{(\mathbb{T}^d)^2}\!K_{j-1}(y_1)K_{j-1}(y_2)\prod_{q=1}^3\eta_{2(r_2-\alpha_q)}(y_1-y_2)\\
&\phantom{\leq}\,\mathrm{d}y_1\,\mathrm{d}y_2)\,\mathrm{d}\alpha_1\mathrm{d}\alpha_2\,\mathrm{d}\alpha_3\\
&\leq c_{1}\sum_{(m_1,m_2,m_3) \in (\mathbb{Z}^d)^3}|\rho_{j-1}(m_1+m_2+m_3)|^2\\
& \phantom{\leq}\int_{[0,r_1]^3}\!(\prod_{q=1}^3\mathrm{e}^{-|2\pi m_q|^4(r_2-\alpha_q)}-\prod_{n=1}^3\mathrm{e}^{-|2\pi m_n|^4(r_1-\alpha_n)})^2\,\mathrm{d}\alpha_1\,\mathrm{d}\alpha_2\,\mathrm{d}\alpha_3
\end{align*}
\begin{align*}
&+c_{1}\sum_{(m_1,m_2,m_3) \in (\mathbb{Z}^d)^3}|\rho_{j-1}(m_1+m_2+m_3)|^2\int_{[0,r_2-r_1]^3}\!\prod_{q=1}^3\mathrm{e}^{-2\alpha_q|2\pi m_q|^4}\,\mathrm{d}\alpha_1\,\mathrm{d}\alpha_2\,\mathrm{d}\alpha_3\\
&\leq c_{2}|r_2-r_1|^{3\gamma}\sum_{(m_1,m_2,m_3) \in (\mathbb{Z}^d)^3}|\rho_{j-1}(m_1+m_2+m_3)|^2\prod_{q=1}^3\frac{1}{(1+|m_q|^4)^{1-\gamma}}\\
&\leq c_2|r_2-r_1|^{3\gamma}\sum_{(m_1,m_2) \in (\mathbb{Z}^d)^2}|\rho_{j-1}(m_1+m_2)|^2(1+|m_1|^4)^{\gamma-1}\sum_{m_3 \in \mathbb{Z}^d}(1+|m_2-m_3|^4)^{\gamma-1}(1+|m_3|^4)^{\gamma-1}\\
&\leq c_3|r_2-r_1|^{3\gamma}\sum_{(m_1,m_2) \in (\mathbb{Z}^d)^2}|\rho_{j-1}(m_1+m_2)|^2(1+|m_1|^4)^{\gamma-1}(1+|m_2|^4)^{\frac{d}{4}+2\gamma-2}\\
&\leq c_3|r_2-r_1|^{3\gamma}\sum_{m_1 \in \mathbb{Z}^d}|\rho_{j-1}(m_1)|^2\sum_{m_2 \in \mathbb{Z}^d}(1+|m_1-m_2|^4)^{\gamma-1}(1+|m_2|^4)^{\frac{d}{4}+2\gamma-2}\\
&\leq c_{4}|r_2-r_1|^{3\gamma}\sum_{m \in \mathbb{Z}^d}|\rho_{j-1}(m)|^2(1+|m|^4)^{\frac{d}{2}+3\gamma-3}\\
&\leq c_{5}2^{(j-1)(3d+12\gamma-12)}|r_2-r_1|^{3\gamma}
\end{align*}
Applying Theorem \ref{kolmogorov}, we complete the proof.
\item For all $\epsilon \in \mathbb{R}_+^*,$ there exist a stochastic process $(\wideparen{G}_{q,\epsilon})_{q \in \mathbb{R}_+}$ of state space $(\mathscr{S}',\mathcal{B}(\mathscr{S}'))$ such that for every $(r,\phi) \in \mathbb{R}_+\times \mathscr{S},\mathds{P}(\wideparen{G}_{r,\epsilon}(\phi)=G_{r,\epsilon}(\phi))=1.$

We fix $(\beta,\gamma) \in \mathbb{R}_+^*\times ]0,1-\frac{d}{6}[.$

It follows from Theorem \ref{nelson} and Lemma \ref{lemma2} that there exists $(c_1,c_2,c_3) \in (\mathbb{R}_+^*)^3$ such that for all $(r_1,r_2,\epsilon,j,x) \in (\mathbb{R}_+)^2\times \mathbb{R}_+^*\times \mathbb{N}\times \mathbb{T}^d,$ if $r_1\leq r_2,$ then 
\begin{align*}
\mathds{E}[|\Delta_{j-1}\wideparen{G}_{r_2,\epsilon}(x)
-\Delta_{j-1}\widetilde{G}_{r_2}(x)\\-\Delta_{j-1}\wideparen{G}_{r_1,\epsilon}(x)+\Delta_{j-1}\widetilde{G}_{r_1}(x)|^\beta]^{\frac{2}{\beta}}&\leq c_1\mathds{E}[|\Delta_{j-1}\wideparen{G}_{r_2,\epsilon}(x)-\Delta_{j-1}\widetilde{G}_{r_2}(x)
-\Delta_{j-1}\wideparen{G}_{r_1,\epsilon}(x)\\
&\phantom{\leq}\ +\Delta_{j-1}\widetilde{G}_{r_1}(x)|^2]\\
&\leq c_1\sum_{(m_1,m_2,m_3) \in (\mathbb{Z}^d)^3}|\rho_{j-1}(m_1+m_2+m_3)|^2\\
&\phantom{\leq}\int_{[0,r_1]^3}\!(\prod_{q=1}^3(\zeta(\epsilon m)-1)^2\mathrm{e}^{-|2\pi m_q|^4(r_2-\alpha_q)}\\
&\phantom{\leq}-\prod_{n=1}^3(\zeta(\epsilon m_n)-1)^2\mathrm{e}^{-|2\pi m_n|^4(r_1-\alpha_n)})^2\,\mathrm{d}\alpha_1\,\mathrm{d}\alpha_2\,\mathrm{d}\alpha_3\\
& \phantom{\leq}\ +c_1\sum_{(m_1,m_2,m_3) \in (\mathbb{Z}^d)^3}|\rho_{j-1}(m_1+m_2+m_3)|^2\\
&\phantom{\leq}\int_{[0,r_2-r_1]^3}\!\prod_{q=1}^3(\zeta(\epsilon m_q)-1)^2\mathrm{e}^{-2\alpha_q|2\pi m_q|^4}\,\mathrm{d}\alpha_1\,\mathrm{d}\alpha_2\,\mathrm{d}\alpha_3\\
&\leq c_2\epsilon^{3\gamma}|r_2-r_1|^{\frac{3\gamma}{2}}\\
&\phantom{\leq} \sum_{(m_1,m_2,m_3) \in (\mathbb{Z}^d)^3}|\rho_{j-1}(m_1+m_2+m_3)|^2\prod_{q=1}^3\frac{1}{(1+|m_q|^4)^{1-\gamma}}\\
&\leq c_32^{(j-1)(3d+12\gamma-12)}\epsilon^{3\gamma}|r_2-r_1|^{\frac{3\gamma}{2}}.
\end{align*}
Applying Theorem \ref{kolmogorov}, we conclude the proof.
\item We deduce from Theorem \ref{kolmogorov} that for every $U \in \mathbb{R}_+,\lim_{\epsilon\to0}\mathds{E}[\sup_{r \in [0,U]}\Vert\widetilde{G}_{r,\epsilon}-\widetilde{G}_r\Vert_{\theta}]=0.$
\end{enumerate}
\end{proof}
\section{Solving the stochastic Cahn-Hilliard equation}
\label{section4}
Finally, we are ready to solve our equation.
\begin{proposition}
\label{proposition1}
If $d \in \{1,2,3\},\alpha \in ]0,2-\frac{d}{2}[,$ and $(g_1,g_2) \in \mathscr{C}^\alpha\times (\bigcap_{U \in \mathbb{R}_+}C_U\mathscr{C}^\alpha),$ then there exists $\beta \in ]0,+\infty]$ such that $$\exists ! f \in \bigcap_{U \in [0,\beta[}C_U\mathscr{C}^\alpha,\forall r \in [0,\beta[,f(r)=P_rg_1+g_2(r)+\int_0^r\!P_{r-q}(\Delta((f(q))^3-f(q)))\,\mathrm{d}q$$ and $$\beta \in \mathbb{R}_+^*\implies \lim_{q\to\beta}\Vert f(q)\Vert_{\alpha}=+\infty.$$
\end{proposition}
\begin{proof}
Let $\gamma:=2(\Vert g_1\Vert_{\alpha}+\Vert g_2\Vert_{C_1\mathscr{C}^\alpha}).$

For all $U \in \mathbb{R}_+,$ let $\Upsilon_{U,g_1,g_2}:C_U\mathscr{C}^\alpha\longrightarrow C_U\mathscr{C}^\alpha$ such that for every $\phi \in C_U\mathscr{C}^\alpha,$
\begin{align*}
\Upsilon_{U,g_1,g_2}(\phi):[0,U]&\longrightarrow \mathscr{C}^\alpha\\
r&\longmapsto (\Upsilon_{U,g_1,g_2}(\phi))(r)=P_rg_1+g_2(r)+\int_0^r\!P_{r-q}(\Delta((\phi(q))^3-\phi(q)))\,\mathrm{d}q
\end{align*}
It follows from Theorem \ref{derivative}, Corollary \ref{product}, and Theorem \ref{Schauder} that there exists $c \in \mathbb{R}_+^*$ such that for any $U \in \mathbb{R}_+$ and all $(\phi_1,\phi_2) \in (C_U\mathscr{C}^\alpha)^2,$ $$\Vert\Upsilon_{U,g_1,g_2}(\phi_1)\Vert_{C_U\mathscr{C}^\alpha}\leq \Vert g_1\Vert_{\alpha}+\Vert g_2\Vert_{C_U\mathscr{C}^\alpha}+c\max(U,U^{\frac{1}{2}})(\Vert \phi_1\Vert_{C_U\mathscr{C}^\alpha}^3+\Vert\phi_1\Vert_{C_U\mathscr{C}^\alpha})$$ and $$\Vert\Upsilon_{U,g_1,g_2}(\phi_2)-\Upsilon_{U,g_1,g_2}(\phi_1)\Vert_{C_U\mathscr{C}^\alpha}\leq c\max(U,U^{\frac{1}{2}})(\Vert\phi_1\Vert_{C_U\mathscr{C}^\alpha}^2+\Vert\phi_2\Vert_{C_U\mathscr{C}^\alpha}^2+1)\Vert\phi_2-\phi_1\Vert_{C_U\mathscr{C}^\alpha}.$$
Hence there exists $\theta \in ]0,1]$ such that for every $(\phi_1,\phi_2) \in (C_\theta\mathscr{C}^\alpha)^2,$ if $$\max(\Vert\phi_1\Vert_{C_\theta\mathscr{C}^\alpha},\Vert\phi_2\Vert_{C_\theta\mathscr{C}^\alpha})\leq\gamma,$$ then $$\Vert\Upsilon_{\theta,g_1,g_2}(\phi_1)\Vert_{C_\theta\mathscr{C}^\alpha}\leq \gamma$$ and $$\Vert\Upsilon_{\theta,g_1,g_2}(\phi_2)-\Upsilon_{\theta,g_1,g_2}(\phi_1)\Vert_{C_\theta\mathscr{C}^\alpha}\leq \frac{1}{2}\Vert\phi_2-\phi_1\Vert_{C_\theta\mathscr{C}^\alpha}.$$
Therefore $$\exists !h  \in C_\theta\mathscr{C}^\alpha,\Upsilon_{\theta,g_1,g_2}(h)=h.$$
Iterating this construction, we complete the proof.
\end{proof}
We deduce from Proposition \ref{proposition1} the following theorem.
\begin{theorem}
If $d \in \{1,2,3\},\alpha \in ]0,2-\frac{d}{2}[,$ and $g \in \mathscr{C}^\alpha,$ then there exist an explosion time $\tau:\Omega\longrightarrow ]0,+\infty]$ and a unique stochastic process $(f_r)_{r \in [0,\tau[}$ of state space $(\mathscr{C}^\alpha,\mathcal{B}(\mathscr{C}^\alpha))$ such that for every $\omega \in \Omega,$ the function $\begin{aligned}[t]
[0,\tau(\omega)[&\longrightarrow \mathscr{C}^\alpha\\
r &\longmapsto f_r(\omega)
\end{aligned}$ is continuous on $[0,\tau(\omega)[$ and $$\forall U \in [0,\tau(\omega)[,f_U(\omega)=P_Ug+\widetilde{X}_U(\omega)+\int_0^U\!P_{U-q}(\Delta((f_q(\omega))^3-f_q(\omega)))\,\mathrm{d}q.$$
\end{theorem}
For all $\epsilon \in \mathbb{R}_+,$ let
\begin{align*}
\psi_\epsilon:\mathbb{R}_+&\longrightarrow \mathbb{R}_+\\
r&\longmapsto \psi_\epsilon(r)=r+\sum_{q \in \mathbb{Z}^d-\{0\}}\frac{1}{2|2\pi q|^4}(\zeta(\epsilon q))^2(1-\mathrm{e}^{-2r|2\pi q|^4})
\end{align*}
We note that $$\forall (r,\epsilon,x) \in \mathbb{R}_+\times \mathbb{R}_+^*\times \mathbb{T}^d,\mathds{E}[|\widetilde{X}_{r,\epsilon}(x)|^2]=\psi_{\epsilon}(r).$$
\begin{proposition}
\label{proposition2}
If $(d,\alpha) \in \{4\}\times ]0,2[$ and $(g_1,g_2,g_3,g_4,g_5,g_6,g_7,g_8) \in (\mathscr{C}^\alpha)^2\times (\bigcap_{(U,\epsilon) \in \mathbb{R}_+\times \mathbb{R}_+^*}C_U\mathscr{C}^{-\epsilon})^6,$ then there exists $(c,\gamma_1,\gamma_2,\beta_1,\beta_2) \in (\mathbb{R}_+^*)^2\times ]0,1[\times ]0,+\infty]^2$ such that $$\exists !f_1 \in \bigcap_{U \in [0,\beta_1[}C_U\mathscr{C}^\alpha,\forall r \in [0,\beta_1[,f_1(r)=P_rg_1+\int_0^r\!P_{r-q}(\Delta((f_1(q))^3+3(f_1(q))^2g_2(q)+3f(q)g_3(q)+g_4(q)$$ $$-f_1(q)-g_2(q)))\,\mathrm{d}q,$$ $$\beta_1 \in \mathbb{R}_+^*\implies \lim_{q\to \beta_1}\Vert f_1(q)\Vert_{\alpha}=+\infty,$$ $$\exists ! f_2 \in \bigcap_{U \in [0,\beta_2[}C_U\mathscr{C}^\alpha,\forall r \in [0,\beta_2[,f_2(r)=P_rg_5+\int_0^r\!P_{r-q}(\Delta((f_2(q))^3+3(f_2(q))^2g_6(q)+3f_2(q)g_7(q)+g_8(q)$$ $$-f_2(q)-g_6(q)))\,\mathrm{d}q,$$ $$\beta_2 \in \mathbb{R}_+^* \implies \lim_{q\to\beta_2}\Vert f_2(q)\Vert_\alpha=+\infty,$$ and for every $(\sigma,U) \in \mathbb{R}_+\times [0,\min(\beta_1,\beta_2)[,$ if $\max(\Vert f_1\Vert_{C_U\mathscr{C}^\alpha},\Vert f_2\Vert_{C_U\mathscr{C}^\alpha})\leq \sigma,$ then $$\Vert f_2-f_1\Vert_{C_U\mathscr{C}^\alpha}\leq (\Vert g_5-g_1\Vert_{\alpha}+c\max(U,U^{1-\gamma_2})((\sigma^2+1)\Vert g_6-g_2\Vert_{C_U\mathscr{C}^{-\gamma_1}}+\sigma \Vert g_7-g_3\Vert_{C_U\mathscr{C}^{-\gamma_1}}+\Vert g_8-g_4\Vert_{C_U\mathscr{C}^{-\gamma_1}}))$$ $$\mathrm{e}^{c\max(U,U^{1-\gamma_2})(2\sigma^2+2\sigma \Vert g_6\Vert_{C_U\mathscr{C}^{-\gamma_1}}+\Vert g_7\Vert_{C_U\mathscr{C}^{-\gamma_1}}+1)}$$
\end{proposition}
\begin{proof}
For all $(U,h_1,h_2,h_3,h_4) \in \mathbb{R}_+\times \mathscr{C}^\alpha\times (\bigcap_{(r,\epsilon) \in \mathbb{R}_+\times \mathbb{R}_+^*}C_r\mathscr{C}^{-\epsilon})^3,$ let $\Upsilon_{U,h_1,h_2,h_3,h_4}:C_U\mathscr{C}^\alpha\longrightarrow C_U\mathscr{C}^\alpha$ such that for every $\phi \in C_U\mathscr{C}^\alpha,$
\begin{align*}
\Upsilon_{U,h_1,h_2,h_3,h_4}(\phi):[0,U]&\longrightarrow \mathscr{C}^\alpha\\
r&\longmapsto 
\begin{aligned}[t]
(\Upsilon_{U,h_1,h_2,h_3,h_4}(\phi))(r)&=P_rh_1+\int_0^r\!P_{r-q}(\Delta((\phi(q))^3+3(\phi(q))^2h_2(q)\\
&\phantom{=}\ +3\phi(q)h_3(q)
+h_4(q)-\phi(q)-h_2(q)))\,\mathrm{d}q
\end{aligned}
\end{align*}
Let $(h_1,h_2,h_3,h_4) \in \{(g_1,g_2,g_3,g_4),(g_5,g_6,g_7,g_8)\},\gamma_1:=\frac{1}{2} \min(\alpha,2-\alpha),\gamma_2:=\frac{1}{4}(2+\alpha+\gamma_1),$ and $\gamma_3:=3(\Vert h_1\Vert_{\alpha}+\Vert h_4\Vert_{C_1\mathscr{C}^{-\gamma_1}}).$

It follows from Theorem \ref{derivative}, Corollary \ref{product}, and Theorem \ref{Schauder} that there exists $c_1 \in \mathbb{R}_+^*$ such that for any $U \in \mathbb{R}_+$ and all $(\phi_1,\phi_2) \in (C_U\mathscr{C}^\alpha)^2,$ $$\Vert \Upsilon_{U,h_1,h_2,h_3,h_4}(\phi_1)\Vert_{C_U\mathscr{C}^\alpha}\leq \Vert h_1\Vert_{\alpha}+c_1\max(U,U^{1-\gamma_2})(\Vert \phi_1\Vert_{C_U\mathscr{C}^\alpha}^3+\Vert h_2\Vert_{C_U\mathscr{C}^{-\gamma_1}}(\Vert \phi_1\Vert_{C_U\mathscr{C}^\alpha}^2+1)$$ $$+\Vert \phi_1\Vert_{C_U\mathscr{C}^\alpha}(\Vert h_3\Vert_{C_U\mathscr{C}^{-\gamma_1}}+1)+\Vert h_4\Vert_{C_U\mathscr{C}^{-\gamma_1}})$$ and $$\Vert\Upsilon_{U,h_1,h_2,h_3,h_4}(\phi_2)-\Upsilon_{U,h_1,h_2,h_3,h_4}(\phi_1)\Vert_{C_U\mathscr{C}^\alpha}\leq c_1\max(U,U^{1-\gamma_2})\Vert \phi_2-\phi_1\Vert_{C_U\mathscr{C}^\alpha}(\Vert \phi_1\Vert_{C_U\mathscr{C}^\alpha}^2+\Vert \phi_2\Vert_{C_U\mathscr{C}^\alpha}^2$$ $$+\Vert h_2\Vert_{C_U\mathscr{C}^{-\gamma}}(\Vert \phi_1\Vert_{C_U\mathscr{C}^\alpha}+\Vert \phi_2\Vert_{C_U\mathscr{C}^\alpha})+\Vert h_3\Vert_{C_U\mathscr{C}^{-\gamma}}+1
).$$
Hence there exists $\theta \in ]0,1]$ such that for every $(\phi_1,\phi_2) \in (C_\theta\mathscr{C}^\alpha)^2,$ if $$\max(\Vert\phi_1\Vert_{C_\theta\mathscr{C}^\alpha},\Vert\phi_2\Vert_{C_\theta\mathscr{C}^\alpha})\leq\gamma_3,$$ then $$\Vert\Upsilon_{\theta,h_1,h_2,h_3,h_4}(\phi_1)\Vert_{C_\theta\mathscr{C}^\alpha}\leq \gamma_3$$ and $$\Vert\Upsilon_{\theta,h_1,h_2,h_3,h_4}(\phi_2)-\Upsilon_{\theta,h_1,h_2,h_3,h_4}(\phi_1)\Vert_{C_\theta\mathscr{C}^\alpha}\leq \frac{1}{2}\Vert\phi_2-\phi_1\Vert_{C_\theta\mathscr{C}^\alpha}.$$
Therefore $$\exists ! \chi\in C_\theta\mathscr{C}^\alpha,\Upsilon_{\theta,h_1,h_2,h_3,h_4}(\chi)=\chi.$$
Iterating this construction we deduce that there exists $(\beta_1,\beta_2) \in ]0,+\infty]^2$ such that $$\exists !f_1 \in \bigcap_{U \in [0,\beta_1[}C_U\mathscr{C}^\alpha,\forall r \in [0,\beta_1[,f_1(r)=P_rg_1+\int_0^r\!P_{r-q}(\Delta((f_1(q))^3+3(f_1(q))^2g_2(q)+3f(q)g_3(q)$$ $$+g_4(q)-f_1(q)-g_2(q)))\,\mathrm{d}q,$$ $$\beta_1 \in \mathbb{R}_+^*\implies \lim_{q\to \beta_1}\Vert f_1(q)\Vert_{\alpha}=+\infty,$$ $$\exists ! f_2 \in \bigcap_{U \in [0,\beta_2[}C_U\mathscr{C}^\alpha,\forall r \in [0,\beta_2[,f_2(r)=P_rg_5+\int_0^r\!P_{r-q}(\Delta((f_2(q))^3+3(f_2(q))^2g_6(q)+3f_2(q)g_7(q)$$ $$+g_8(q)-f_2(q)-g_6(q)))\,\mathrm{d}q,$$ and $$\beta_2 \in \mathbb{R}_+^* \implies \lim_{q\to\beta_2}\Vert f_2(q)\Vert_\alpha=+\infty.$$
It follows from Theorem \ref{derivative}, Corollary \ref{product}, Lemma \ref{lemma}, and Theorem \ref{Schauder} that there exists $c_2 \in \mathbb{R}_+^*$ such that for every $(\sigma,U) \in \mathbb{R}_+\times [0,\min(\beta_1,\beta_2)[,$ if $\max(\Vert f_1\Vert_{C_U\mathscr{C}^\alpha},\Vert f_2\Vert_{C_U\mathscr{C}^\alpha})\leq\sigma,$ then 
\begin{align*}
\forall r \in [0,U],\Vert f_2(r)-f_1(r)\Vert_{\alpha} &\leq \Vert(\Upsilon_{U,g_5,g_6,g_7,g_8}(f_2))(r)-(\Upsilon_{U,g_5,g_6,g_7,g_8}(f_1))(r)+P_rg_5\\
&\phantom{\leq}\ -P_rg_1+\int_0^r\!P_{r-q}(\Delta((f_1(q))^2(g_6(q)-g_2(q))+f_1(q)(g_7(q)-g_3(q))\\
&\phantom{\leq} +g_8(q)-g_4(q))+g_6(q)-g_2(q))\Vert_{\alpha}\\
&\leq \Vert g_5-g_1\Vert_{\alpha}+c_2\max(U,U^{1-\gamma_2})((\sigma^2+1)\Vert g_6-g_2\Vert_{C_U\mathscr{C}^{-\gamma_1}}+\sigma \Vert g_7-g_3\Vert_{C_U\mathscr{C}^{-\gamma_1}}\\
&\phantom{\leq}+\Vert g_8-g_4\Vert_{C_U\mathscr{C}^{-\gamma_1}})+c_2(2\sigma^2+2\sigma \Vert g_6\Vert_{C_U\mathscr{C}^{-\gamma_1}}+\Vert g_7\Vert_{C_U\mathscr{C}^{-\gamma_1}}+1)\\
&\phantom{\leq}\ \int_0^r\!\max(|r-q|^{-\gamma_2},1)\Vert f_2(q)-f_1(q)\Vert_{C_U\mathscr{C}^\alpha}\,\mathrm{d}q.
\end{align*}
Hence for all $(n,\sigma,U) \in \mathbb{N}\times \mathbb{R}_+\times [0,\min(\beta_1,\beta_2)[,$ if $\max(\Vert f_1\Vert_{C_U\mathscr{C}^\alpha},\Vert f_2\Vert_{C_U\mathscr{C}^\alpha})\leq \sigma,$ then $$\Vert f_2-f_1\Vert_{C_U\mathscr{C}^\alpha}\leq (\Vert g_5-g_1\Vert_{\alpha}+c_2\max(U,U^{1-\gamma_2})((\sigma^2+1)\Vert g_6-g_2\Vert_{C_U\mathscr{C}^{-\gamma_1}}+\sigma \Vert g_7-g_3\Vert_{C_U\mathscr{C}^{-\gamma_1}}$$ $$+\Vert g_8-g_4\Vert_{C_U\mathscr{C}^{-\gamma_1}}))\sum_{m=0}^n\frac{1}{m!}(c_2(2\sigma^2+2\sigma \Vert g_6\Vert_{C_U\mathscr{C}^{-\gamma_1}}+\Vert g_7\Vert_{C_U\mathscr{C}^{-\gamma_1}}+1)\int_0^U\!\max(1,q^{-\gamma_2})\,\mathrm{d}q)^m$$ $$+\frac{2\sigma}{(n+1)!}(c_2(2\sigma^2+2\sigma \Vert g_6\Vert_{C_U\mathscr{C}^{-\gamma_1}}+\Vert g_7\Vert_{C_U\mathscr{C}^{-\gamma_1}}+1)\int_0^U\!\max(1,q^{-\gamma_2})\,\mathrm{d}q)^{n+1}.$$
Letting $n\to+\infty$, we conclude the proof.
\end{proof}
We end this section by proving local well-posedness of the $4$-dimensional stochastic Cahn-Hilliard equation.
\begin{theorem}
If $(d,\alpha) \in \{4\}\times \mathbb{R}_-^*$ and $g \in \mathscr{C}^\alpha,$ then
\begin{enumerate}
\item there exist an explosion time $\tau:\Omega\longrightarrow ]0,+\infty]$ and a unique stochastic process $(h_r)_{r \in [0,\tau[}$ of state space $(\mathscr{C}^\alpha,\mathcal{B}(\mathscr{C}^\alpha))$ such that for every $\omega \in \Omega,$ the function $\begin{aligned}[t]
[0,\tau(\omega)[&\longrightarrow \mathscr{C}^\alpha\\
r &\longmapsto h_r(\omega)
\end{aligned}$ is continuous on $[0,\tau(\omega)[$ and $$\forall U \in [0,\tau(\omega)[,h_U(\omega)=P_Ug+\int_0^U\!P_{U-q}(\Delta((h_q(\omega))^3+3(h_q(\omega))^2\widetilde{X}_q(\omega)+3h_q(\omega)\widetilde{Y}_q(\omega)+\widetilde{G}_q(\omega)$$ $$-h_q(\omega)-\widetilde{X}_q(\omega)))\,\mathrm{d}q.$$
\item for every $\epsilon \in \mathbb{R}_+^*,$ there exist an explosion time $\tau_\epsilon:\Omega\longrightarrow ]0,+\infty]$ and a unique stochastic process $(f_{r,\epsilon})_{r \in [0,\tau_\epsilon[}$ of state space $(\mathscr{C}^\alpha,\mathcal{B}(\mathscr{C}^\alpha))$ such that for all $\varkappa \in \Omega,$ the function $\begin{aligned}[t]
[0,\tau_\epsilon(\varkappa)[&\longrightarrow \mathscr{C}^\alpha\\
r &\longmapsto f_{r,\epsilon}(\varkappa)
\end{aligned}$ is continuous on $[0,\tau_\epsilon(\varkappa)[$ and for $\mathds{P}$-almost every $\omega \in \Omega,$ $$\forall U \in [0,\tau_\epsilon(\omega)[,f_{U,\epsilon}(\omega)=P_U(\zeta_\epsilon*g)+\widetilde{X}_{U,\epsilon}(\omega)+\int_0^U\!P_{U-q}(\Delta((f_{q,\epsilon}(\omega))^3-3f_{q,\epsilon}(\omega)\psi_\epsilon(q)-f_{q,\epsilon}(\omega)))$$ $$\,\mathrm{d}q.$$
\item there exist a sequence $(\epsilon_m)_{m \in \mathbb{N}}$ in $\mathbb{R}_+^*$ and a sequence $(\nu_m)_{m \in \mathbb{N}}$ of finite stopping times such that $\lim_{q\to+\infty}\epsilon_q=0,\lim_{q\to+\infty}\nu_q=\tau,$ and $$\forall \sigma \in \mathbb{R}_+^*,\lim_{q\to+\infty}\mathds{P}(\sup_{r \in [0,\nu_q]}\Vert f_{r,\epsilon_q}-h_r-\widetilde{X}_r\Vert_{\alpha}>\sigma)=0.$$
\end{enumerate}
\end{theorem}
\begin{proof}
\begin{enumerate}
\item It follows from Proposition \ref{proposition2} that there exist an explosion time $\tau:\Omega\longrightarrow ]0,+\infty]$ and a unique stochastic process $(h_r)_{r \in [0,\tau[}$ of state space $(\mathscr{C}^\alpha,\mathcal{B}(\mathscr{C}^\alpha))$ such that for every $\omega \in \Omega,$ the function $\begin{aligned}[t]
[0,\tau(\omega)[&\longrightarrow \mathscr{C}^\alpha\\
r &\longmapsto h_r(\omega)
\end{aligned}$ is continuous on $[0,\tau(\omega)[$ and $$\forall U \in [0,\tau(\omega)[,h_U(\omega)=P_Ug+\int_0^U\!P_{U-q}(\Delta((h_q(\omega))^3+$$$$3(h_q(\omega))^2\widetilde{X}_q(\omega)+3h_q(\omega)\widetilde{Y}_q(\omega)+\widetilde{G}_q(\omega)-h_q(\omega)-\widetilde{X}_q(\omega)))\,\mathrm{d}q.$$
\item We deduce from Proposition \ref{proposition2} that for every $\epsilon \in \mathbb{R}_+^*,$ there exist an explosion time $\tau_\epsilon:\Omega\longrightarrow ]0,+\infty]$ and a unique stochastic process $(h_{r,\epsilon})_{r \in [0,\tau_\epsilon[}$ of state space $(\mathscr{C}^\alpha,\mathcal{B}(\mathscr{C}^\alpha))$ such that for every $\omega \in \Omega,$ the function $\begin{aligned}[t]
[0,\tau_\epsilon(\omega)[&\longrightarrow \mathscr{C}^\alpha\\
r &\longmapsto h_{r,\epsilon}(\omega)
\end{aligned}$ is continuous on $[0,\tau_\epsilon(\omega)[$ and $$\forall U \in [0,\tau_\epsilon(\omega)[,h_{U,\epsilon}(\omega)=P_U(\zeta_\epsilon*g)+\int_0^U\!P_{U-q}(\Delta((h_{q,\epsilon}(\omega))^3+3(h_{q,\epsilon}(\omega))^2\widetilde{X}_{q,\epsilon}(\omega)$$ $$+3h_{q,\epsilon}(\omega)\widetilde{Y}_{q,\epsilon}(\omega)+\widetilde{G}_{q,\epsilon}(\omega)-h_{q,\epsilon}(\omega)-\widetilde{X}_{q,\epsilon}(\omega)))\,\mathrm{d}q.$$
Noticing that for $\mathds{P}$-almost every $\omega \in \Omega$ and every $U \in [0,\tau_\epsilon(\omega)[,$ $$h_{U,\epsilon}(\omega)+\widetilde{X}_{U,\epsilon}=P_U(\zeta_\epsilon*g)+\widetilde{X}_{U,\epsilon}(\omega)+\int_0^U\!P_{U-q}(\Delta((h_{q,\epsilon}(\omega)+\widetilde{X}_{q,\epsilon}(\omega))^3$$ $$-3(h_{q,\epsilon}(\omega)+\widetilde{X}_{q,\epsilon}(\omega))\psi_\epsilon(q)-h_{q,\epsilon}(\omega)-\widetilde{X}_{q,\epsilon}(\omega)))\,\mathrm{d}q,$$ we complete the proof.
\item For all $(m,\epsilon) \in \mathbb{N}\times \mathbb{R}_+^*,$ let $\upsilon_{m,\epsilon}:=\min(\inf\{r \in [0,\min(\tau,\tau_\epsilon)[,\max(\sup_{q \in [0,r]}\Vert h_q\Vert_\alpha,$ $\sup_{q \in [0,r]}\Vert h_{q,\epsilon}\Vert_\alpha)\geq m\},m).$

It follows from Lemma \ref{lemma}, Theorem \ref{theorem1}, Theorem \ref{theorem2}, Theorem \ref{theorem3}, and Proposition \ref{proposition2} that $$\forall (m,\sigma) \in \mathbb{N}\times \mathbb{R}_+^*,\lim_{\epsilon\to0}\mathds{P}(\sup_{r \in [0,\upsilon_{m,\epsilon}]}\Vert h_{r,\epsilon}+\widetilde{X}_{r,\epsilon}-h_r-\widetilde{X}_r\Vert_{\alpha}>\sigma)=0.$$
Consequently, there exists a sequence $(\epsilon_q)_{q \in \mathbb{N}}$ in $\mathbb{R}_+^*$ such that $\lim_{m\to+\infty}\epsilon_m=0,\lim_{m\to+\infty}\upsilon_{m,\epsilon_m}=\tau,$ and $$\forall \sigma \in \mathbb{R}_+^*,\lim_{m\to+\infty}\mathds{P}(\sup_{r \in [0,\upsilon_{m,\epsilon_m}]}\Vert h_{r,\epsilon_m}+\widetilde{X}_{r,\epsilon_m}-h_r-\widetilde{X}_r\Vert_{\alpha}>\sigma)=0.$$
\end{enumerate}
\end{proof}

\textit{Department of Mathematics and Statistics, University of Ottawa, Ottawa,
K1N 6N5, Canada}

\textit{Email: jghaf099@uottawa.ca}
\end{document}